\documentclass[12pt,oneside,a4paper,reqno]{amsart}
\usepackage{amssymb,amsmath,amsthm,amsfonts}
\usepackage[T1]{fontenc}
\usepackage{tgtermes}

\usepackage[english]{babel}
\usepackage{enumerate,esint}

\usepackage[dvipsnames]{xcolor}
\definecolor{ourcolor}{RGB}{0,102,204}
\usepackage{hyperref}
 \hypersetup{
 colorlinks,
 citecolor=ourcolor,
 linkcolor=ourcolor,
 urlcolor=black}

\numberwithin{equation}{section}

\usepackage[left=2.5cm,right=2.5cm,top=2.5cm,bottom=2.5cm]{geometry}
 \makeatletter
 \def\@seccntformat#1{\hspace*{0mm}%
  \protect\textup{\protect\@secnumfont
    \ifnum\pdfstrcmp{subsection}{#1}=0 \bfseries\fi
    \csname the#1\endcsname
    \protect\@secnumpunct
      }%
 }

\usepackage{mathtools}
\usepackage{graphicx}

\usepackage{epstopdf}
\allowdisplaybreaks

\theoremstyle{plain}
\newtheorem{thm}{Theorem}[section]
\newtheorem{lemma}[thm]{Lemma}

\newtheorem{cor}[thm]{Corollary}

\theoremstyle{definition}

\theoremstyle{remark}
\newtheorem{remark}[thm]{Remark}

\usepackage{todonotes}
\usepackage{euscript} 

\usepackage{enumitem}
\usepackage{esint} 

\DeclareMathOperator*{\essinf}{ess\,inf}

\DeclareUnicodeCharacter{043E}{\textendash}
\newcommand*{\bigchi}{\mbox{\Large$\chi$}}

\newcommand\QOm{Q_\Omega}
\newcommand\QBR{Q_{B_R}}

\newcommand\R{\mathbb{R}}
\newcommand\Sf{\mathbb{S}}
\newcommand{\tmSep}{{; }}

\title[Nonlocal Micromagnetics]{Nonlocal Micromagnetics: Compactness Criteria, Existence of Minimizers, and Brown's Fundamental Theorem}

\author[G. D. Fratta] {Giovanni Di Fratta}
\author[R. Giorgio]{Rossella Giorgio}
\author[L. Lombardini]{Luca Lombardini}	

\date{}
\AtEndDocument{\bigskip
	\hrule
	\medskip
         \small{
		\noindent
		{\sc G. Di Fratta},
		Dipartimento di Matematica e Applicazioni ``R. Caccioppoli'',
		Universit\`a degli Studi di Napoli ``Federico II'',
		Via Cintia, Complesso Monte S. Angelo, 80126 Naples, Italy.
		
		\noindent
		E-mail:
		\href{mailto:giovanni.difratta@unina.it}{\tt giovanni.difratta@unina.it}.
	}
	
	\medskip \medskip \medskip \hrule  \medskip
    \small{
		\noindent
		{\sc R. Giorgio},
		Institute of Analysis and Scientific Computing,
		TU Wien,
		Wiedner Hauptstra\ss e 8-10, 1040 Vienna, Austria
           {\sc and}
            \noindent MedUni Wien, Währinger Gürtel 18-20, 1090 Vienna, Austria.
            
		\noindent
		E-mail:
		\href{mailto:rossella.giorgio@tuwien.ac.at}{\tt rossella.giorgio@tuwien.ac.at}.
	}
	\medskip \medskip \medskip \hrule  \medskip
    \small{
		\noindent
		{\sc L. Lombardini},
        Institut f\"{u}r Mathematik, Fakult\"{a}t f\"{u}r Mathematik, Universit\"{a}t Wien, Oskar-Morgenstern-Platz 1, 1090 Vienna, Austria.

        \noindent
		E-mail:
		\href{mailto:luca.lombardini@univie.ac.at}{\tt luca.lombardini@univie.ac.at}.
        }
\medskip
\hrule
}

\begin{document}

\begin{abstract}
This paper investigates the existence and qualitative properties of minimizers for a class of nonlocal micromagnetic energy functionals defined on bounded domains. The considered energy functional consists of a symmetric exchange interaction, which penalizes spatial variations in magnetization, and a magnetostatic self-energy term that accounts for long-range dipolar interactions. Motivated by the extension of Brown's fundamental theorem on fine ferromagnetic particles to nonlocal settings, we develop a rigorous variational framework in \( L^2(\Omega ; \mathbb{S}^2) \) under mild assumptions on the interaction kernel \( j \), including symmetry, Lévy-type integrability, and prescribed singular behavior.  For spherical domains, we generalize Brown’s fundamental results by identifying critical radii \( R^* \) and \( R^{**} \) that delineate distinct energetic regimes: for \( R \leq R^* \), the uniform magnetization state is energetically preferable (\emph{small-body regime}), whereas for \( R \geq R^{**} \), non-uniform magnetization configurations become dominant (\emph{large-body regime}). These transitions are analyzed through Poincaré-type inequalities and explicit energy comparisons between uniform and vortex-like magnetization states.  

Our results directly connect classical micromagnetic theory and contemporary nonlocal models, providing new insights into domain structure formation in nanoscale magnetism. Furthermore, the mathematical framework developed in this work contributes to advancing theoretical foundations for applications in spintronics and data storage technologies.  
\end{abstract}



{\maketitle}
	

\section{Introduction and Motivation}
The study of nonlocal energy functionals is central in
mathematical physics, offering a powerful framework for modeling systems
governed by long-range interactions. These functionals are particularly
relevant in materials science, where phenomena such as micromagnetics, phase transitions, and elasticity demand a nuanced treatment of interactions across
multiple scales. In the context of ferromagnetic materials, the interplay between short-range quantum mechanical effects and long-range magnetostatic forces dictates the formation of complex magnetization patterns---domain
walls, vortices, and skyrmions---that hold significant potential for
applications in spintronics and high-density data storage technologies.

This work focuses on the variational analysis of a nonlocal micromagnetic
energy functional defined on bounded domains. Specifically, we investigate the
energy functional
\begin{equation}
  \label{eq:mainfunJpW} \mathcal{E}_{\Omega} (m) := \mathcal{J}_{\Omega}
  (m) +\mathcal{W}_{\Omega} (m),
\end{equation}
where the terms $\mathcal{J}_{\Omega}$ and $\mathcal{W}_{\Omega}$ are defined
as
\begin{equation}
  \mathcal{J}_{\Omega} (m) := \iint_{\Omega \times \Omega} j (x - y) 
  \lvert m (x) - m (y) \rvert^2  \hspace{0.17em} dx \hspace{0.17em} dy,
  \label{eq:intro_exchange}
\end{equation}
and
\begin{equation}
  \mathcal{W}_{\Omega} (m) := - \int_{\Omega} h_d [m] (x) \cdot m (x) 
  \hspace{0.17em} dx. \label{eq:intro_magnetostatic}
\end{equation}
Here, $\mathcal{J}_{\Omega} (m)$ represents a symmetric nonlocal exchange
interaction, while $\mathcal{W}_{\Omega} (m)$ accounts for the so-called
magnetostatic self-energy. The functional $\mathcal{E}_\Omega$ is defined over the space $L^2
(\Omega ; \mathbb{S}^2)$ of square-integrable magnetization fields $m : \Omega
\to \mathbb{S}^2$ on a bounded domain $\Omega \subset \mathbb{R}^3$,
reflecting the physical constraint of local saturation.

Our work is motivated by {\emph{Brown's fundamental theorem of fine
ferromagnetic particles}}~{\cite{brown1969fundamental}}, a cornerstone in the
variational theory of micromagnetism. Brown's theorem establishes a critical
domain size below which uniform magnetization minimizes the classical
micromagnetic energy. Specifically, for ferromagnetic particles shaped as
spheres or triaxial ellipsoids, the ground state is uniformly magnetized if
the particle's diameter is below a critical size (\textit{small bodies
regime}). For larger particles, long-range dipole-dipole interactions
dominate, leading to non-uniform configurations, such as vortex-like magnetization patterns
(\textit{large bodies regime}). While Brown's result has already been rigorously formulated
in Sobolev spaces, where the Dirichlet energy approximates nonlocal
Heisenberg-type energies in the limit of short-range interactions, its
extension to nonlocal settings remains unexplored. This gap is particularly
relevant given the growing interest in nonlocal models, which better capture
nanoscale phenomena where long-range interactions and geometric confinement
play decisive roles.

A primary objective of this study is to identify physically meaningful
conditions on the \textit{exchange kernel} $j$ in \eqref{eq:intro_exchange}
that enable the development of a robust variational framework for
micromagnetism directly in $L^2  (\Omega ; \mathbb{S}^2)$. Specifically, we
aim to establish assumptions on $j$ that not only preserve the
phenomenological insights of Brown's theorem for spherical domains but also
ensure sufficient compactness, thereby enhancing the model's practical
applicability for analyzing long-range interactions in complex geometries.
This approach is pivotal for extending the predictive power of classical
micromagnetic models to more general, nonlocal settings. Moreover, our
proposed framework is highly relevant for emerging applications in nanoscale
magnetism, where nonlocal effects and intricate domain structures critically
shape magnetic behavior.

In this work, we begin by establishing compactness results for the energy
functional $\mathcal{E}_{\Omega}$, where $\Omega$ is a generic bounded open
set, thereby ensuring the existence of minimizers. The compactness result is
presented in a more abstract setting, which applies to a broader class of
nonlocal exchange functionals $\mathcal{J}_{\Omega}$ and to the case where
$\mathcal{W}_{\Omega}$ is replaced by a functional that is bounded from below
and continuous for the strong topology in $L^2$. Additionally, we
extend the analysis to a more general framework, where $\Omega$ is a bounded
open subset of $\R^n$ and the maps $m$ take values in a compact set
$\mathcal{N} \subseteq \R^{\ell}$. These results provide the necessary
foundations for the subsequent investigation of the energy functional in
specific domains, ensuring the necessary conditions for the existence of
minimizers.

Following this, we focus on domains $\Omega$ shaped as balls with radius $R$.
We establish that Brown's theorem remains valid in the nonlocal setting under
physically reasonable assumptions on the exchange kernel $j$. Specifically, we
prove the existence of a critical radius $R^{\ast}$ such that, for $R \leq
R^{\ast}$ (\textit{small bodies regime}), all minimizers of the nonlocal
energy functional $\mathcal{E}_{\Omega}$ exhibit spatially uniform
magnetization. Conversely, when $R \geq R^{\ast \ast}$, for a second threshold
radius $R^{\ast \ast} \geq R^{\ast}$ (\textit{large bodies regime}), the
ground states transition to non-uniform configurations, often characterized by
vortex-like structures. This behavior underscores the intricate interplay
between exchange interactions, which favor uniform magnetization and
magnetostatic interactions, which drive spatial variation, resulting in a
delicate energy balance that governs the observed transitions.

Before stating our results, we set the stage by introducing the proper
functional background and providing a brief review of the physical context
necessary to understand how our results fit into the variational theory of
micromagnetism. This includes a discussion of the nonlocal exchange
interaction, the magnetostatic self-energy, and the physical principles
underlying Brown's theorem. By unifying these elements, we aim to provide a
comprehensive framework for analyzing nonlocal micromagnetic energy
functionals and their dependence on domain size.

\subsection*{Outline of the paper}The structure of the paper is as follows.
The remainder of this section provides an overview of the physical framework
that motivated our study. In Section~\ref{Sec:MainResults}, we review the
state of the art in relation to our contributions, present a detailed
formulation of the problem, and state the main results. Sections~\ref{Sec:cmpt
ed exist} and~\ref{Sec:proofs-mag} are devoted to the proofs. Specifically,
Section \ref{Sec:cmpt ed exist} focuses on the compactness argument and the
existence result for the general nonlocal model, while Section~\ref{Sec:proofs-mag} extends Brown's theorem to the nonlocal setting.
\subsection{Physics context: Micromagnetics}\label{subs:magnetostatic-energy}
Micromagnetics is a continuum
theory rooted in the works of of Landau-Lifshitz~{\cite{LandauA1935}} and
Brown~{\cite{BrownB1963,BrownB1962}}, which describes magnetic phenomena at
the mesoscopic scale (1--1000 nm). At this intermediate resolution, for a
rigid ferromagnetic particle occupying a region $\Omega \subseteq \R^3$, the
magnetization is modeled as a vector field $M : \Omega \to \mathbb{R}^3$ with
constant magnitude $M_s$ (spontaneous magnetization), normalized here to
unity.
The theory balances competing energy contributions: exchange interactions
favoring alignment, magnetostatic forces promoting flux closure, and
anisotropy terms tied to crystalline structure. In this work, we focus on the
interplay between nonlocal exchange and magnetostatic energies, which govern
domain structure formation in the absence of lower-order contributions like
magnetocrystalline anisotropy and external fields, whose contributions do not
affect neither the existence of minimizers (see Corollary \ref{Cor:anisotropy+DMI}) nor the domain-dependence phenomena, since the leading contributions result to be the two nonlocal terms
$\mathcal{J}_{\Omega}$ and $\mathcal{W}_{\Omega}$ (cf. Brown
in~{\cite{brown1969fundamental}}).\textbf{}
In this setting, the observable magnetization states are the minimizers of the
\textit{micromagnetic energy functional} $\mathcal{E}_{\Omega} =
\mathcal{J}_{\Omega}  +\mathcal{W}_{\Omega}$ introduced in
\eqref{eq:mainfunJpW}, defined on the metric space $L^2  (\Omega ; \Sf^2)$.

The \emph{nonlocal} exchange interaction $\mathcal{J}_{\Omega} $ in \eqref{eq:intro_exchange} penalizes
non-uniformities in the magnetization orientation. The kernel $j :
\mathbb{R}^3 \to \mathbb{R}_+$ encodes interaction strength between magnetic
moments at positions $x$ and $y$, with symmetry ($j (z) = j (- z)$) ensuring
energy conservation. Physically, $\mathcal{J}_{\Omega} $ arises as the
continuum limit of discrete Heisenberg-type interactions, where $j (x - y)$
corresponds to exchange constants between atomic spins. Under localization ($j
(z) \sim \rho_{\varepsilon} (z) / |z|^2$ as $\varepsilon \to 0$),
$\mathcal{J}_{\Omega} $ recovers the classical Dirichlet energy, as shown
by Bourgain, Brezis, and Mironescu in~{\cite{bourgain2001another}}.

The \emph{magnetostatic self-energy} $\mathcal{W}_{\Omega}$ quantifies the
long-range dipolar interactions among the magnetic moments. For a
magnetization $m \in L^2  (\Omega ; \mathbb{R}^3)$ the magnetostatic
self-energy  $\mathcal{W}_{\Omega} (m)$ is defined by \eqref{eq:intro_magnetostatic}. In the expression of $\mathcal{W}_{\Omega} (m)$,
the field $h_d [m]$ (commonly referred to as the {\emph{demagnetizing
field}} or stray field) is characterized as the unique solution in $L^2 
(\mathbb{R}^3 ; \mathbb{R}^3)$ of the Maxwell--Ampére equations of
magnetostatics (see, e.g.,~{\cite{di2020variational}}):
\begin{equation}
  \label{eq:Maxwell} \left\{\begin{array}{ll}
    \mathrm{div} (m \chi_{\Omega} + h_d [m]) = 0 & \textrm{in }
    \mathbb{R}^3,\\
    \mathrm{curl} \hspace{0.17em} h_d [m] = 0 & \textrm{in } \mathbb{R}^3 .
  \end{array}\right.
\end{equation}
In the above, $m \chi_{\Omega}$ denotes the extension of $m$ by zero outside
$\Omega$, and $h_d$ acts as a nonlocal operator on $L^2  (\mathbb{R}^3 ;
\mathbb{R}^3)$. It is well known that $h_d$ is a bounded, self-adjoint
operator with norm one and is negative definite; indeed, one may show that
\begin{equation}
  \label{term:magnetostatic} \mathcal{W}_{\Omega} (m) = \int_{\mathbb{R}^3} |
  h_d [m] (x) |^2  \hspace{0.17em} dx,
\end{equation}
for all $m \in L^2  (\Omega ; \mathbb{R}^3)$.

In the particular case where $\Omega$ is a bounded Lipschitz domain and $m \in
H^1 (\Omega ; \mathbb{S}^2)$, the unique solution of \eqref{eq:Maxwell} can be
expressed as $h_d [m] = - \nabla \Phi$, with $\Phi \in H^1 (\mathbb{R}^3)$
defined for $x \in \Omega$ by
\begin{equation}
  \label{eq:Phi} \Phi (x) = \frac{1}{4\pi} \left\{ \int_{\Omega} \frac{- \mathrm{div} m (y)}{|x - y|}
  \hspace{0.17em} dy + \int_{\partial \Omega} \frac{m (y) \cdot n (y)}{|x -
  y|}  \hspace{0.17em} dS (y) \right\} .
\end{equation}
Here, $n$ denotes the outward unit normal on $\partial \Omega$, and for
simplicity, we omitted mathematically irrelevant constants. For further
properties and detailed discussions regarding the demagnetizing field operator
and the magnetostatic self-energy, we refer the reader
to~{\cite{di2020variational,MR2094598}}.

\section{State of the art and contributions of the present work} \label{Sec:MainResults}
\subsection{State of the art}
Nonlocal models have emerged as a prominent area of interest in applied
mathematics, largely because of their effectiveness in capturing intricate
behaviors, microstructural changes, and interactions across different
geometric scales, all while remaining aligned with experimental observations.
These models find extensive application in continuum mechanics and phase
transitions, as they inherently accommodate long-range interactions. While
there is extensive literature on nonlocal models; this discussion will focus
on specific studies that are particularly relevant to our findings and
applications in micromagnetics.

In the classical theory of micromagnetism, the exchange energy is local and
given by the classical Dirichlet energy $\int_{\Omega} | \nabla m|^2$. A cornerstone of the theory is the so-called \textit{Brown's fundamental theorem
of fine ferromagnetic particles} {\cite{brown1969fundamental}}: there exists a
critical diameter $d_c$ such that, for $d < d_c$, uniform magnetization states
are the sole {\emph{global}} minimizers of the micromagnetic free energy
functional. The result explains the high coercivity observed in fine
ferromagnetic particles: in sufficiently small single-domain particles, where
inter-particle magnetostatic interactions are negligible; magnetization
reversal occurs through rigid rotation instead of domain wall displacement,
resulting in enhanced coercivity. In its rigorous form, in {\cite{brown1969fundamental}}, the result has been rigorously established for
spherical particles, whereas real materials often possess elongated
geometries. Aharoni expanded upon Brown's findings to include prolate
spheroids~{\cite{Aharoni_1988}}, and further extensions to general triaxial
ellipsoidal particles have been provided in {\cite{di2012generalization}} (see
also {\cite{Alouges_2004}}) where also an estimate of $d_c$ is given in terms
of the so-called demagnetizing factors of the general ellipsoid
{\cite{di2016newtonian}}. The result was further extended in
{\cite{alouges2015liouville}}, where it was proven that {\emph{local}}
minimizers are necessarily spatially uniform if the ferromagnetic particle
occupies a sufficiently small ellipsoidal region, as well as a quantitative estimate of the minimal exchange energy in sufficiently
small (uniformly) convex particles, which justify the Stoner-Wohlfarth approximation.

In the context of nonlocal interactions, to the best of authors' knowledge,
Rogers introduced the first model for symmetric exchange energy in
{\cite{rogers1991nonlocal}}; in there, the author proposes a new model designed
to describe macroscopic effects in ferromagnetic materials in which the
nonlocal term
\begin{equation}
  \mathcal{E}_{{nl}} (m) = \iint_{\Omega \times \Omega} m (x) \cdot m (y) k (x - y)
  \hspace{0.17em} d  x \hspace{0.17em} d  y,
\end{equation}
where, up to material-dependent constants, the kernel $k (x - y) = e^{- \gamma |x - y|} / |x - y|$ replaces the
Dirichlet energy. Also, in there, it is proved an existence result for
measure-valued magnetizations minimizing the nonlocal exchange energy. The drawback is that there is no way of ensuring that the limit of minimizing
sequence satisfies the nonconvex constraint of $m$ being
$\mathbb{S}^2$-valued, i.e., the weak limit of the minimizing sequence is not
a solution to the problem, and one considers Young's
measure associated with the sequence as a solution of the minimization
problem. The results in {\cite{rogers1991nonlocal}} are inspired by the work
of {\cite{James_1990}}, who studied measure-valued minimizers within the
context of the theory when the exchange energy is omitted; the \ study \
reveals that even in the absence of exchange energy, minimizing sequences can
accurately model many features of observed domain structure.\textbf{}

One of the advantages of these nonlocal models is that they avoid the
coercivity paradox described by Brown (see {\cite{BrownB1963}}). This paradox
stems from the observation that these models tend to predict excessively wide hysteresis loops. In {\cite{Brandon_1992}}, the authors demonstrate
that the nonlocal model introduced in {\cite{rogers1991nonlocal}} does not
suffer from the coercivity paradox.

In this direction, a central challenge in the study of nonlocal micromagnetics
is the identification of suitable mathematical conditions on the exchange
kernels that guarantee a well-posed variational framework in the weaker
setting of $L^2 (\Omega, \mathbb{S}^2)$. Addressing this issue, our work
extends Brown’s theorem to a nonlocal setting, thereby establishing a crucial theoretical link between classical micromagnetics and modern
variational approaches to nonlocal interactions that remain consistent with
physical observations. This extension provides a rigorous foundation for the
class of admissible kernels required to guarantee the existence and
qualitative properties of minimizers of nonlocal micromagnetic energy
functionals, offering new insights into the influence of long-range
interactions on magnetic domain structures.

Our existence theorem is based on a compactness argument in $L^2 (\Omega ;
\mathcal{N})$, where $\mathcal{N}$ is a compact set, that generalizes the
findings of~{\cite{jarohs2020weth}}, which established compactness in
$L^2_{\mathrm{loc}} (\R^n)$ under the same kernel assumptions. A related
compactness result was previously obtained in~{\cite{correa2018nonlocal}},
where additional assumptions on the interaction kernel---such as radial
symmetry, positivity, and controlled growth near the origin---were imposed
(see conditions (H1) and (H2) in~{\cite{correa2018nonlocal}}).

The interaction kernels considered in this work, as well as in the aforementioned studies, belong to the class of Lévy kernels. They can be considered a generalization of the \textit{fractional} kernel (see, e.g.,~\cite{2012hitchhiker}), and indeed, in the literature, Lévy kernels are frequently associated with nonlocal operators in
the study of partial differential equations (PDEs), including Dirichlet and
Neumann problems (see, e.g.,~{\cite{felsinger2015dirichlet,foghem2022general}}
and references therein). However, the nonlocal energies corresponding to these
operators differ slightly from those considered here, as they are typically
defined for functions in the entire space $\R^n$ but subject to nonlocal
boundary conditions.

In contrast, energies of the type \eqref{eq:intro_exchange} have received
less attention in this context and are often analyzed from a probabilistic
perspective, particularly in connection with censored Lévy processes (see,
e.g.,~{\cite{bogdan2003censored}}). These processes, which involve jumps and
heavy tails are particularly suited for modeling phenomena where abrupt
changes or outliers are significant.


Finally, we mention that a similar variational framework has been explored in
the study of nonlocal problems arising in peridynamics, where the governing
equations are formulated using nonlocal integral operators rather than
classical local differential operators (see,
e.g.,~{\cite{mengesha2014bond,mengesha2015variational,mengesha2023optimal}}).


\subsection{Contributions of the present work I: Existence of minimizers}
A key objective of this work is to establish a rigorous variational framework for nonlocal micromagnetics, focusing on the interplay between exchange interactions and magnetostatic self-energy. For that, we introduce a general class of nonlocal functionals characterized by an interaction kernel satisfying precise analytical properties.

We consider an exchange kernel \( j: \mathbb{R}^n \to [0,+\infty] \) that adheres to the following assumptions:
\begin{enumerate}[label=(J\arabic*)]
    \item \label{J1} \textbf{Symmetry}: The kernel is symmetric, meaning
    \[
    j(-z) = j(z) \quad \text{for all } z \in \mathbb{R}^n.
    \]
    \item\label{J2} \textbf{Lévy-type Condition}: The kernel satisfies
    \[
    L_j \coloneqq \int_{\mathbb{R}^n} \min\{1, |z|^2\} \, j(z) \, dz < +\infty.
    \]
    \item \label{J3}\textbf{Non-integrability Condition}: To ensure the nonlocal nature of the interaction, 
    the kernel satisfies
    \[
    j \notin L^1(B_R) \quad \text{for every } R > 0,
    \]
    where \( B_R \) denotes the ball of radius \( R \) centered at the origin.
\end{enumerate}

\noindent For a compact set $\mathcal N\subseteq\R^\ell$ and a bounded open set $\Omega\subseteq\R^n$, we define the function spaces $
    L^2 (\Omega; \mathcal N):= \{ m \in L^2 (\Omega; \R^\ell) : m \in \mathcal N \text{ a.e.\ in } \Omega \}$
and 
\begin{equation}
    \mathcal X^j(\Omega;\mathcal N)\coloneqq\left\{m\in L^2(\Omega;\mathcal N)\,:\,\mathcal{J}_{\Omega}(m) <+\infty\right\},
\end{equation}
where the nonlocal interaction energy $\mathcal{J}_{\Omega}$ is given by \eqref{eq:intro_exchange}.
The notation extends to cases where \( m \) is not constrained to take values in the compact set \( \mathcal{N} \). Specifically, we use \( \mathcal{X}^j(\Omega; \mathbb{R}^\ell) \) to denote the space of functions in \( L^2(\Omega; \mathbb{R}^\ell) \) with finite \( \mathcal{J}_\Omega \)-energy.

To consider more general interaction mechanisms that can be of some relevance in physical contexts, even different from  micromagnetics (see also Remark~\ref{rmk:motpsi}),
we introduce a function $\psi : \mathcal{N} \times \mathcal{N} \to [0, +
\infty)$ and a measurable kernel $K : \Omega \times \Omega \to [0, + \infty]$
satisfying, for some constant $\Lambda \geq 1$, the following \textit{growth conditions}:
\begin{equation}\label{eq:control_psi}
    \frac{1}{\Lambda}|p-q|^2\leq\psi(p,q)\leq\Lambda|p-q|^2\qquad\textrm{for all }p,q\in\mathcal N,
\end{equation}
and 
 \begin{equation}\label{eq:control_kappa}
    \frac{1}{\Lambda}j(x-y)\leq K(x,y)\qquad\textrm{for a.e.\ }x,y\in\Omega. 
     \end{equation}
These conditions ensure that $K$ controls the nonlocal interaction strength while $\psi$ can be tough as modeling specific geometric or energetic properties of the target manifold. Then, for every $m\in\mathcal X^j(\Omega;\mathcal N)$ we define the functionals 
\begin{equation}\label{functional:K}
\mathcal K_\Omega(m)\coloneqq
\iint_{\Omega\times\Omega}K(x,y)|m(x) -m(y)|^2 \,dx\,dy
\end{equation}
and 
\begin{equation}\label{func:psi}
\mathcal F_\Omega(m)\coloneqq
\iint_{\Omega\times\Omega}K(x,y)\,\psi(m(x),m(y)) \,dx\,dy.
\end{equation}
From \eqref{eq:control_psi} and \eqref{eq:control_kappa}, we deduce the energy bounds
\begin{equation}\label{eq:en_domain}
\mathcal{J}_\Omega(m) \leq \Lambda \mathcal K_\Omega(m) \leq \Lambda^2 \mathcal F_\Omega(m) \leq \Lambda^3 \mathcal K_\Omega(m),
\end{equation}
which establish a fundamental connection between these functionals: $\mathcal F_\Omega$ bounded is equivalent to having $\mathcal K_\Omega$ bounded, and both situations imply a bound on $ \mathcal{J}_{\Omega}$. In particular, the functional 
$\mathcal F_\Omega(m)$ is finite if and only if $m$ belongs to the space
\begin{equation*}
    \mathcal X^K(\Omega;\mathcal N)\coloneqq\left\{m\in L^2(\Omega;\mathcal N)\,:\, \mathcal K_\Omega(m) <+\infty \right\},
\end{equation*}
and we have the inclusion $\mathcal X^K(\Omega;\mathcal N)\subseteq\mathcal X^j(\Omega;\mathcal N)$. 

Our first result is a compactness theorem, which is crucial for the variational analysis of the minimizers of the nonlocal energy functionals we are interested in; in primis, the energy functional \eqref{eq:mainfunJpW} issuing from the variational theory of micromagnetism.
\begin{thm}[Compactness]\label{thm:Compactness}
    Let $(m_k)_k\subseteq \mathcal X^K(\Omega;\mathcal N)$ be such that
\begin{equation}\label{eq:nonloso}
    \sup_{h\in\mathbb N}\mathcal F_\Omega(m_h)<+\infty.
    \end{equation}
Then, there exist a subsequence $(m_{k_h})_h$ and a function $m\in \mathcal X^K(\Omega;\mathcal N)$ such that
\begin{equation*}
\lim_{h\to\infty}\|m_{k_h}-m\|_{L^2(\Omega;\R^\ell)}=0.
\end{equation*}
\end{thm}
\noindent This result extends the compactness arguments developed in \cite[Theorem 1.1]{jarohs2020weth}, providing a foundational step for establishing existence theorems for minimizers of nonlocal micromagnetic energies. The proof is presented in
Section \ref{subsec:cmptc}.

\begin{remark}[On the necessity of Assumption \ref{J3}]\label{rmk:non-integrable}
Assumption \ref{J3} (which ensures the non-integrability of $j$) plays a key role in establishing strong compactness. Indeed, as observed in~\cite{jarohs2020weth}, if $j \in L^1(\R^n)$, then for any $m \in L^2(\Omega; \R^\ell)$, the nonlocal energy satisfies the estimate  
\begin{equation*}
\iint_{\Omega\times\Omega}|m(x)-m(y)|^2 j(x-y)\,dx\,dy  \leq  4 \, \|m \|^{2}_{L^2(\Omega; \R^\ell)} \| j\|_{L^1(\R^n)}.
\end{equation*} 
As a consequence, the space $\mathcal X^j(\Omega;\R^\ell)$ reduces to $L^2(\Omega; \R^\ell)$, leading to a loss of compactness.  
\end{remark}

\begin{remark}[About Assumption \ref{J2}]
The L\'evy-type condition \ref{J2} places two requirements on the kernel $j$: first, it prevents $j$ from being too singular near zero, and second, it ensures that $j$ is integrable away from the origin, including at infinity. If we restrict our attention to a fixed bounded open set $\Omega$ of diameter $R$, we may equivalently replace \ref{J2} with
\[
\int_{B_R} |z|^2 j(z) dz < +\infty.
\]
However, the integrability of $j$ at infinity becomes essential when examining the behavior of minimizers in balls of growing radius, as treated in Theorem~\ref{thm:large_bodies} (see also Lemma~\ref{lemma:j6}).
\end{remark}

\begin{remark}[Motivation for the function $\psi$]\label{rmk:motpsi}
The motivation for considering a general function $\psi$ satisfying \eqref{eq:control_psi}, rather than $|p-q|^2$, arises when the target space $\mathcal{N}$ is a compact Riemannian submanifold of $\mathbb{R}^\ell$. In this setting, $\mathcal{N}$ is naturally endowed with its intrinsic geodesic distance $d_{\mathcal{N}}$, which is equivalent to the Euclidean distance induced by the ambient space, meaning that there exists a constant $C_{\mathcal{N}} \geq 1$ such that
\begin{equation}\label{eq:geo_dist} |p-q|\leq d_{\mathcal{N}}(p,q)\leq C_{\mathcal{N}} |p-q| \qquad \text{for all } p,q\in\mathcal{N}. \end{equation}
This equivalence suggests that choosing $\psi = d_{\mathcal{N}}^2$ in the definition of the nonlocal energy $\mathcal{F}_\Omega$ provides a more faithful measure of interaction between values $m(x)$ and $m(y)$, preserving the intrinsic geometry of $\mathcal{N}$.
A similar reasoning applies whenever a compact subset $\mathcal{N} \subseteq \mathbb{R}^\ell$ is equipped with a distance $d_{\mathcal{N}}$ that is equivalent to the Euclidean distance.
\end{remark}

Our second result concerns the existence of minimizers for the nonlocal energy functional \eqref{eq:mainfunJpW}. In fact, we can prove the existence of minimizers for a broader class of nonlocal energy functionals. Specifically, we consider the following functional:
\begin{equation}\label{func:total}
    \mathcal{I}_{\Omega}(m) : = \mathcal{F}_{\Omega}(m) + \mathcal{P}_{\Omega}(m) = \iint_{\Omega\times\Omega}K(x,y)\,\psi(m(x),m(y))\,dx\,dy \, + \, \mathcal{P}_{\Omega}(m), 
\end{equation} 
which is defined for every $m \in \mathcal{X}^K (\Omega ; \mathcal{N})$. The
first term represents the functional $\mathcal{F}_{\Omega}$ introduced in
\eqref{func:psi}, while the second term $\mathcal{P}_{\Omega} : L^2 (\Omega ;
\mathcal{N}) \to \mathbb{R}$ is a functional that is bounded from below and
continuous for the strong topology in $L^2$.

Minimizing the functional $\mathcal{F}_{\Omega}$ by itself is trivial, as it
does not require any conditions on the kernel $K$ (other than its
non-negativity and its being finite almost everywhere), since
$\mathcal{F}_{\Omega} \geq 0$ and $\mathcal{F}_{\Omega} (m) = 0$ whenever $m$ is a constant configuration.
On the other hand, given the non-convex constraint of $m$ being $\mathcal{N}$-valued, minimizing the functional $\mathcal{I}_{\Omega}$, which is
the sum of the nonlocal interaction $\mathcal{F}_{\Omega}$ and the energy
$\mathcal{P}_{\Omega}$ (of which the magnetostatic self-energy
$\mathcal{P}_{\Omega} =\mathcal{W}_{\Omega}$, as defined in
\eqref{eq:intro_magnetostatic}, is an example), requires additional considerations.
In this context, we establish the following general existence result, whose proof is provided in Section
\ref{subsec:existence} as a direct consequence of the compactness result stated in Theorem \ref{thm:Compactness}.

\begin{thm}[Existence of minimizers]\label{thm:existence}
The functional $\mathcal{I}_{\Omega}$ defined in \eqref{func:total}, admits at least a minimizer in the class $\mathcal{X}^K(\Omega;\mathcal N).$  
\end{thm}
It is worth noting that the existence of a minimizer for
$\mathcal{I}_{\Omega}$, at least in some specific cases, can be established
through alternative methods, without relying on the direct method of the
Calculus of Variations, and thus without invoking Theorem
\ref{thm:Compactness}. An example of this approach is given in Theorem
\ref{prop:exist-constanti}.

\begin{remark}
  The continuity of the functional $\mathcal{P}_{\Omega}$ with respect to the
  strong topology in $L^2$ can be replaced by continuity in $L^p$ for any $p
  \in [1, \infty)$, since $\Omega$ is bounded and $m \in \mathcal{N}$ almost
  everywhere.
\end{remark}

\subsection{Contributions of the present work II: Applications to Micromagnetics}\label{subsec:micromag_results}
In this section, we explore the feasibility of
developing a variational framework for Micromagnetics within the solely $L^2$ setting. Specifically, we focus on the nonlocal model in three dimensions ($n=\ell=3$). Let $\Omega \subseteq \R^3$ be a bounded domain, and consider the target set $\mathcal{N} = \Sf^2$ which encodes the locally saturated constraint imposed by Micromagnetic
Theory.

For each $m \in \mathcal{X}^{j}(\Omega; \Sf^2)$, we study the associated energy functional $\mathcal{E}_{\Omega}(m)=\mathcal{J}_{\Omega}(m) + \mathcal{W}_{\Omega}(m)$ given by \eqref{eq:mainfunJpW}, 
where we recall, the term $\mathcal{J}_\Omega(m)$
represents the nonlocal exchange energy functional as defined in \eqref{eq:intro_exchange}, while the term $\mathcal{W}_{\Omega}(m)$  
represents the magnetostatic self-energy already introduced in \eqref{eq:intro_magnetostatic}. This latter term aligns well with
our existence result (see Theorem~\ref{thm:existence}), as it is always non-negative and continuous for the strong topology in $L^2$ (see Section \ref{subs:magnetostatic-energy}).

The interaction kernel $j$ is always assumed to satisfy Assumptions \ref{J1}-\ref{J3}. In specific parts of the analysis, additional conditions on $j$ are required to derive sharper results.
We list them below.
\begin{enumerate}[label=(J\arabic*)]
    \setcounter{enumi}{3}
    \item\label{J4} We assume that
    \begin{equation}\label{cond:inf}
    \QOm : = \essinf_{|z|< \operatorname{diam}(\Omega)}  j(z) > 0,
     \end{equation}
    \item \label{J5} There exist $R_0 > 0$, $C > 0$ and $ s \in (0,1)$ such that
    \begin{equation}
    j(z) \geq \frac{C}{|z|^{3 + 2s}} \quad \textrm{for every } 
    z \in B_{R_0}.
    \end{equation}
    where, we recall, $B_{R_0}$ denotes the ball centered at the origin with radius $R_0$.
\end{enumerate}
Later on, we specify in which cases Assumptions \ref{J4}-\ref{J5} are needed.
\begin{remark} 
In Assumption \ref{J5}, we require a lower bound within a given ball provided by the function $|z|^{-(3 + 2s)}$, which is precisely the kernel that appears in the definition of the classical $H^s$-seminorm. It is important to note that Assumption \ref{J5} is stronger than the non-integrability condition in Assumption \ref{J3}, which is the minimal assumption to get strong compactness results in the relevant function spaces. 
Moreover, if $R_0\geq \operatorname{diam}(\Omega)$, then Assumption \ref{J5} clearly implies also Assumption \ref{J4}.    
\end{remark}

The L\'evy-type condition \ref{J2} directly implies the following integral growth estimate.

\begin{lemma}\label{lemma:j6}
    Let $j:\R^n\to[0,+\infty]$ be a Borel-measurable function satisfying condition \ref{J2}. Then
    \[
    \lim_{R\to\infty}\frac{1}{R^2}\int_{B_R}|z|^2 j(z)\, dz=0.
    \]
\end{lemma}
\begin{proof}
    For \( R > 1 \), decompose the integral over \( B_R \) as:
    \[
    \int_{B_R} |z|^2 j(z) \, dz = \int_{B_1} |z|^2 j(z) \, dz + \int_{B_R \setminus B_1} |z|^2 j(z) \, dz.
    \]
    By condition \ref{J2}, the first term is finite. Consequently,
    \[
    \lim_{R \to \infty} \frac{1}{R^2} \int_{B_1} |z|^2 j(z) \, dz = 0.
    \]
    For the second term, define \( u_R(z) := \frac{1}{R^2} \chi_{B_R \setminus B_1}(z) |z|^2 j(z) \). Observe that \( 0 \leq u_R(z) \leq \chi_{\mathbb{R}^n \setminus B_1}(z) j(z) \) for all \( z \in \mathbb{R}^n \), and \( u_R(z) \to 0 \) pointwise as \( R \to \infty \). 
    Since \( \chi_{\mathbb{R}^n \setminus B_1} j \in L^1(\mathbb{R}^n) \) because of assumption \ref{J2}, Lebesgue's Dominated Convergence Theorem implies:
    \[
    \lim_{R \to \infty}  \frac{1}{R^2} \int_{B_R \setminus B_1} |z|^2 j(z) \, dz= \lim_{R \to \infty} \int_{\mathbb{R}^n} u_R(z) \, dz = 0.
    \]
    Combining both limits completes the proof.
\end{proof}

Notice that, by choosing $\mathcal N=\mathbb S^2\subseteq\mathbb R^3, K=j, \psi(p,q)=|p-q|^2$ and $\mathcal P_\Omega=\mathcal W_\Omega$, the functional $\mathcal I_\Omega$ defined in \eqref{func:total} is exactly $\mathcal I_\Omega=\mathcal E_\Omega$. Thus, Theorem~\ref{thm:existence} directly provides the following existence result.

\begin{thm}[Existence of magnetic equilibrium configurations]\label{thm:existence-magn} 
Assume that the kernel $j$ satisfies \ref{J1}--\ref{J3}. 
Then, the micromagnetic energy $\mathcal{E}_\Omega$ admits a minimizer in $\mathcal X^j(\Omega;\mathbb S^2)$.

\end{thm}

As a direct corollary of Theorem~\ref{thm:existence} we can also investigate the existence of minimizers in the presence of additional contributions to the micromagnetic functional. Specifically, we consider the \textit{anisotropy energy}
\begin{equation}\label{eq:anisotropy}
\mathcal{A}_{\Omega}(m):= \int_{\Omega} \varphi(m)\, dx, 
\end{equation}
where $\varphi: \Sf^2 \to \R_{+}$ is a Lipschitz function with some preferred vanishing directions, and the \textit{nonlocal antisymmetric exchange energy} \begin{equation}\label{eq:nonloc-DMI}
   \mathcal{D}_{\Omega}(m) := \iint_{\Omega \times \Omega} \mu(x-y) \cdot (m(x) \times m(y)) \, dx \,dy,
\end{equation}
where $\mu \in L^1(\R^3; \R^3)$ is an odd vector-valued interaction kernel, i.e.\ $ \mu(-z)=-\mu(z)$ for every $z\in\R^3$. The functional $\mathcal{D}_{\Omega}$ represents the nonlocal counterpart of the {\textit{Dzyaloshinskii-Moriya interaction energy}} (see~\cite{davoli2024bourgain} for formal justification) which arises in certain ferromagnetic crystals with broken inversion symmetry.

\begin{cor}[Existence in the presence of the anisotropy energy and the antisymmetric exchange]\label{Cor:anisotropy+DMI}
Assume that the kernel $j$ satisfies \ref{J1}--\ref{J3}. 
Then, the total micromagnetic energy  
\begin{equation*}
    \mathcal{J}_{\Omega}(m) + \mathcal{W}_{\Omega}(m) + \mathcal{A}_{\Omega}(m) + \mathcal{D}_{\Omega}(m)
\end{equation*}
admits a minimizer in $ \mathcal X^j(\Omega; \Sf^2)$.
\end{cor}
\begin{proof}
By Theorem~\ref{thm:existence}, it suffices to prove that the perturbation term
\[
\mathcal{P}_{\Omega}(m) := \mathcal{W}_{\Omega}(m) + \mathcal{A}_{\Omega}(m) + \mathcal{D}_{\Omega}(m), \quad m \in L^2(\Omega;\Sf^2),
\]
is continuous with respect to the strong $L^2(\Omega;\Sf^2)$ topology and bounded from below so that it can be viewed as a continuous perturbation of $\mathcal{J}_{\Omega}$.

We treat each term separately. The micromagnetic self-energy $\mathcal{W}_{\Omega}$ is both bounded and continuous by the properties detailed in Section~\ref{subs:magnetostatic-energy}. The anisotropy energy $\mathcal{A}_{\Omega}$ is nonnegative, and its continuity follows directly from the Lipschitz continuity of the anisotropy density $\varphi$. 
For the antisymmetric exchange term $\mathcal{D}_{\Omega}$, we first note the estimate
\[
\left|\mathcal{D}_{\Omega}(m)\right| \le \|\mu\|_{L^1(\R^3;\R^3)} \|m\|^2_{L^2(\Omega;\Sf^2)} \le |\Omega| \, \|\mu\|_{L^1(\R^3;\R^3)},
\]
where we have used the fact that $|m(x)|=1$ almost everywhere in $\Omega$. Moreover, for any $m_1,m_2\in L^2(\Omega;\Sf^2)$, the integrability of $\mu$ yields
\begin{align*}
\left|\mathcal{D}_{\Omega}(m_1) - \mathcal{D}_{\Omega}(m_2)\right|
&\le \iint_{\Omega\times\Omega} |\mu(x-y)|\, |m_1(x)\times m_1(y) - m_2(x)\times m_2(y) |\,dx\,dy\\[1mm]
&\le 2 \iint_{\Omega\times\Omega} |\mu(x-y)|\,|m_1(x)|\,|m_1(y)-m_2(y)|\,dx\,dy\\[1mm]
&\le 2\,|\Omega|^{\frac{1}{2}} \|\mu\|_{L^1(\R^3;\R^3)}\, \|m_1 - m_2\|_{L^2(\Omega;\Sf^2)}.
\end{align*}
In the second inequality, we have used the antisymmetry of $\mu$, and then applied standard estimates. This shows that $\mathcal{D}_{\Omega}$ is Lipschitz continuous for the strong $ L^2$-topology.
Collecting these estimates, we conclude that $\mathcal{P}_{\Omega}$ is a bounded-from-below, $L^2(\Omega;\Sf^2)$-continuous perturbation of $\mathcal{J}_{\Omega}$, which completes the proof.
\end{proof}

We now return to the functional \(\mathcal{E}_\Omega\) defined in \eqref{eq:mainfunJpW}, which comprises solely the nonlocal exchange term \(\mathcal{J}_\Omega\) (see \eqref{eq:intro_exchange}) and the magnetostatic self‐energy \(\mathcal{W}_\Omega\) (see \eqref{term:magnetostatic}). The subsequent Theorems~\ref{thm:small_bodies} and \ref{thm:large_bodies} characterize the minimizers of \(\mathcal{E}_\Omega\) as a function of the size of the domain \(\Omega\). In our analysis, we focus on spherical domains by setting \(\Omega = B_R\), where \(B_R\) denotes the ball of radius \(R > 0\) centered at the origin. We examine the qualitative behavior of the minimizers for both small and large values of \(R\). Extensions to the case in which \(\Omega\) is a triaxial ellipsoid can be obtained with minor modifications following \cite{di2012generalization}; however, we do not address that case further here.

The proofs of the next two results (Theorems~\ref{thm:small_bodies} and~\ref{thm:large_bodies}) are provided in Section \ref{Sec:proofs-mag}.
\begin{thm}[Small bodies] \label{thm:small_bodies}
Assume that the kernel $j$ satisfies \ref{J1},\ref{J2} and \ref{J5}.
Then there exists a critical radius $ R^{*} = R^{*}(j) \in(0,R_0/2)$ such that, for $
    R \leq R^{*}$,
every minimizer $m_* \in \mathcal X^j(B_R; \Sf^2)$ of the energy $ \mathcal{E}_{B_R}$ defined in \eqref{eq:mainfunJpW} is constant. 
\end{thm}
\begin{thm}[Large bodies]\label{thm:large_bodies}
Assume that the kernel $j$ satisfies 
\ref{J1}--\ref{J3}.
Then there exists a critical radius $ R^{**} = R^{**}(j) > 0 $ such that if $
    R \geq R^{**}$, 
every minimizer $m_* \in \mathcal X^j(B_R; \Sf^2)$ of the energy $ \mathcal{E}_{B_R}$ defined in \eqref{eq:mainfunJpW} is non-constant. 
\end{thm} 
Theorem~\ref{thm:small_bodies} follows as a direct corollary of a more general result, stated below, whose proof is provided in Section~\ref{Sec:proofs-mag}.
\begin{thm}\label{prop:exist-constanti}
Consider $j:\R^3\to[0,+\infty]$ a Borel-measurable function, finite almost everywhere, and satisfying \ref{J4}.
Let $C_R:=1/{(\QBR|B_R|)}> 0$ be the constant of the Poincaré-type inequality \eqref{Ineq:Poincaré}.
If $C_R < 3$, then the energy $\mathcal{E}_{B_R}$ defined in \eqref{eq:mainfunJpW} has a minimizer in $\mathcal X^j(B_R; \Sf^2)$. Moreover, the minimizers are exactly the constant functions. 
\end{thm}

\begin{remark}[On Assumption \ref{J4} in Theorem \ref{thm:small_bodies}]
Hypothesis \ref{J4} is required to establish a Poincaré-type inequality for the nonlocal exchange term \(\mathcal{J}_{\Omega}\) in \eqref{eq:intro_exchange}, where the associated constant explicitly depends on the size of the domain (see Lemma \ref{lemma:Poincaré} for further details). This inequality plays a fundamental role in the proof of Theorem \ref{thm:small_bodies}, as outlined in Section~\ref{Sec:proofs-mag}.
\end{remark}

In general, the existence of minimizers in Theorem \ref{thm:existence-magn} relies on the kernel \( j \) satisfying conditions \ref{J1}--\ref{J3}. In particular, as highlighted in Remark \ref{rmk:non-integrable}, the non-integrability condition \ref{J3} is crucial to ensure strong compactness in \( L^2 \) (see Theorem \ref{thm:Compactness}). On the other hand, Theorem \ref{prop:exist-constanti} presents a scenario in which, even in the absence of assumptions \ref{J1}--\ref{J3}, the existence of minimizers can still be established by leveraging the size-dependent properties of \( \mathcal{E}_{B_R} \) rather than compactness arguments.

\begin{remark}[On the condition \( C_R < 3 \) in Theorem \ref{prop:exist-constanti}]
In Theorem \ref{prop:exist-constanti}, the specific choice of \( j \) significantly influences the range of radii for which minimizers remain constant, leading to a completely different scenario from the \textit{small bodies} regime. This variation arises from the fact that assumptions \ref{J1}--\ref{J3} are no longer imposed.  In particular, the kernel 
$j$ may no longer be singular at the origin.
For instance, if \( j \equiv 1 \), a minimizer exists and is necessarily constant whenever \( R > (4\pi)^{-\frac{1}{3}} \). We thus obtain that the minimizers are constant in the case of \textit{large bodies}. 

On the other hand, for a kernel of the form \( j(z) = 4e^{-|z|^2} \), the range of radii for which minimizers are constant is confined to a bounded interval, approximately \( R \in (0.28, 2.61) \).
\end{remark}

\begin{remark}[Example of the co-existence of both phenomena]
The co-existence of both phenomena, namely the \textit{small bodies regime} in Theorem \ref{thm:small_bodies} and the \textit{large bodies regime} in Theorem \ref{thm:large_bodies}, is certainly ensured for a kernel
$j: \R^3 \to [0, +\infty]$ that is a Borel-measurable, symmetric, and such that for some function $f\in L^1(\mathbb R^3\setminus B_1)$, constant $\lambda \geq 1$, and fractional indexes $s, \sigma \in (0,1)$, with $\sigma\geq s$, there holds
\begin{equation*}
    \frac{1}{\lambda} \frac{1}{|z|^{3+2s}} \bigchi_{B_1}(z) \leq j(z) \leq \lambda\frac{1}{|z|^{3+2\sigma}} \bigchi_{B_1}(z)+f(z)\bigchi_{\R^3 \setminus B_1}(z) \quad \textrm{for every } z \in \R^3,
\end{equation*}
\end{remark}

\begin{remark}[On Assumption \ref{J5} in Theorem \ref{thm:small_bodies}]
Although Assumption \ref{J5} may initially appear restrictive, it arises
naturally from the analytical requirements of Theorem \ref{thm:small_bodies}.
In particular, the proof requires that the exchange kernel $j$ satisfies two
critical conditions. First, one must have
\begin{equation}
    \essinf_{z \in B_{2 R}} j (z) > \frac{1}{4 \pi R^3}  \quad \text{for
   sufficiently small } R, \label{cond:rmk}
\end{equation}
and second, the Poincaré constant $C_R$ in the inequality
\begin{equation} \label{eq:PItemp1}
\|m - \langle m \rangle_{B_R} \|^2_{L^2 (B_R ; \mathbb{R}^3)} \leq C_R
   \iint_{B_R \times B_R} j (x - y)  \hspace{0.17em} |m (x) - m (y) |^2 
   \hspace{0.17em} dx \hspace{0.17em} dy 
\end{equation}
must decay as $R \to 0$. Kernels of the form $j (z) = |z|^{- 3}$ fail to
satisfy this latter requirement. Indeed, a simple scaling argument shows that
the inequality \eqref{eq:PItemp1} can be reformulated as
\[ \|m - \langle m \rangle_{B_R} \|^2_{L^2 (B_R ; \mathbb{R}^3)} \leq
   \frac{C_1}{R^3} \iint_{B_R \times B_R} j \hspace{-0.17em} \left( \frac{x -
   y}{R} \right)  \hspace{0.17em} |m (x) - m (y) |^2  \hspace{0.17em} dx
   \hspace{0.17em} dy, \]
which clearly indicates that with $j (z) = |z|^{- 3}$ the Poincaré constant
becomes $0$-homogeneous (i.e., $C_R = C_1$ is independent of $R$).
Consequently, $|z|^{- 3}$ does not provide the necessary decay in singularity.
Combining this with \eqref{cond:rmk}, we understand that $j$ must satisfy $j (z)
\geq c |z|^{- 3 - 2 s}$ for $R < R_0$, where $c > 0$ and $s \in (0, 1)$. The
exponent $2 s$ strengthens the singularity, while the condition $s < 1$ preserves $j$ as a
Lévy kernel, ensuring compatibility with conditions \ref{J1}--\ref{J3}.
Of course, this argument is a heuristic motivation rather than a rigorous necessity
proof for Assumption \ref{J5},  since condition \eqref{cond:rmk} is obtained by means of an explicit Poincaré constant which
 is not the optimal one.
\end{remark}

\section{Compactness and existence of minimizers (proofs of
Theorem~\ref{thm:Compactness} and Theorem~\ref{thm:existence})}\label{Sec:cmpt ed exist}
The proof of Theorem \ref{thm:Compactness} is an extension of~\cite[Theorem 1.1]{jarohs2020weth}. 
We first state a few technical lemmas.

\begin{lemma}\label{lem:esteso}
Assume that the kernel $j$ satisfies the Assumptions \ref{J1}-\ref{J2}. Let $m\in \mathcal{X}^j(\Omega; \R^\ell)$, i.e., let $m\in L^2(\Omega; \R^\ell)$ be such that
    \begin{equation*}
        \mathcal{J}_{\Omega}( m) = \iint_{\Omega\times\Omega}j(x-y)|m(x)-m(y)|^2 \, dx\, dy<+\infty.
    \end{equation*}
    Also, assume that $m$ vanishes outside an open set and is compactly contained in $\Omega$, i.e., that there exists an open set $\Omega'\Subset\Omega$ such that $m=0$ almost everywhere in $\Omega\setminus\Omega'$. Then, for the extension $\tilde{m}$ defined by
    \begin{equation*}
        \tilde{m}\coloneqq  \begin{cases} m & \textrm{in } \Omega, \\ 0 &  \textrm{in } \R^n\setminus\Omega, \end{cases}  
    \end{equation*}
    the following estimate holds
    \begin{equation*}
        \mathcal{J}_{\R^n}(\tilde{m}) \leq \mathcal{J}_{\Omega}(m) +C\,L_j\,\|m\|_{L^2(\Omega)}^2,
    \end{equation*}
    for some constant $C=C(\Omega',\Omega)>0$, and with $L_j$ being the constant defined by the Lévy-type condition in Assumption \ref{J2}.
\end{lemma}
\begin{remark}
From the proof will be evident that one can take, e.g., $C(\Omega', \Omega) =  \frac{2(1 + d^2)}{d^2}$ where $ d := \operatorname{dist}(\Omega',\partial\Omega).$
\end{remark}

\begin{proof}
    We only need to prove that
    \begin{equation*}
        \int_\Omega\bigg(\int_{\R^n\setminus\Omega}j(x-y)|m(x)|^2 \,dy\bigg)\, dx \leq C\,L_j\,\|m\|_{L^2(\Omega)}^2,
    \end{equation*}
    for some constant $C=C(\Omega',\Omega)>0$.
    For that, let $d:= \operatorname{dist}(\Omega',\partial\Omega)$ and notice that
    \begin{align*}
        \int_\Omega\bigg(\int_{\R^n\setminus\Omega}j(x-y)|m(x)|^2 \,dy\bigg)\, dx&=\int_{\Omega'}|m(x)|^2\bigg(\int_{\R^n\setminus\Omega} j(x-y)\,dy\bigg)\, dx\\
        &
        =\frac{1}{d^2}\int_{\Omega'}|m(x)|^2\bigg(\int_{\R^n\setminus\Omega}\min\{d^2,|x-y|^2\} \, j(x-y)\,dy\bigg)\, dx\\
        &
        \leq\frac{1 + d^2}{d^2}\int_{\Omega'}|m(x)|^2\bigg(\int_{\R^n}\min\{1,|z|^2\}\, j(z)\,dz\bigg)\, dx\\
        &
        \leq \frac{1 + d^2}{d^2}\, L_j\,\|m\|_{L^2(\Omega)}^2,
    \end{align*}
    with $L_j < + \infty$ from \ref{J2}. This concludes the proof.
\end{proof}

\begin{lemma}\label{lem:tagliato} Assume that the kernel $j$ satisfies Assumptions \ref{J1}-\ref{J2}.
     Let $\varphi\in C^{0,1}(\overline{\Omega})$ and let $m\in \mathcal{X}^j(\Omega; \R^\ell)$ i.e., let $m\in L^2(\Omega; \R^\ell)$ be such that
    \begin{equation*}
        \mathcal{J}_{\Omega}(m) = \iint_{\Omega\times\Omega} j(x-y)|m(x)-m(y)|^2\, dx\, dy<+\infty.
    \end{equation*}
     Then
    \begin{align*}
        \mathcal{J}_{\Omega}(\varphi m) \leq C \, (1 +  L_j) \,\|\varphi\|_{C^{0,1}(\bar{\Omega})}^2 \left(\|m\|_{L^2(\Omega)}^2 + \mathcal{J}_{\Omega}(m) \right)
    \end{align*}
    for some constant $C=C(\Omega)>0$, and with $L_j$ being the constant defined by the Lévy-type condition in Assumption \ref{J2}. Here we set
    \begin{equation*}
        \|\varphi\|_{C^{0,1}(\bar{\Omega})} := \|\varphi\|_{C^{0}(\bar{\Omega})} + [\varphi]_{C^{0,1}(\overline{\Omega})}, \quad  \text{with} \quad [\varphi]_{C^{0,1}(\overline{\Omega})} := \sup_{x,y \in \Omega,\, x \ne y} \frac{|\varphi(x) - \varphi(y)|}{|x-y|}.
    \end{equation*}
\end{lemma}
\begin{remark}
From the proof, it will be evident that one can take, e.g., $C(\Omega)=2 \,(1 + \operatorname{diam}^2 (\Omega)). $
\end{remark}

\begin{proof}
First, we note that since $\Omega$ is bounded, for every $y \in \Omega$ there holds
\begin{align*}
  \int_\Omega|x-y|^2j(x-y)\,dx &= \int_{\Omega \cap \{ |x-y| \leq 1 \}} \min\{1,|x-y|^2\} j(x-y)\,dx  \\
  & \qquad\qquad\qquad \qquad + \int_{\Omega \cap \{ |x-y| > 1 \}} |x-y|^2j(x-y)\,dx \\
  &
  \leq \int_{\Omega \cap \{ |x-y| \leq 1 \}} \min\{1,|x-y|^2\} j(x-y)\,dx \\
  &  \qquad\qquad\qquad\qquad + \operatorname{diam}^2 (\Omega)  \int_{\Omega \cap \{ |x-y| > 1 \}} j(x-y)\,dx\\
  &\leq (1 + \operatorname{diam}^2 (\Omega) ) \int_{\Omega} \min\{1,|x-y|^2\} j(x-y)\,dx\\
  & \leq  L_j (1 + \operatorname{diam}^2 (\Omega)),  
\end{align*}
with $L_j < + \infty$ from \ref{J2}. Then,  we estimate $\mathcal{J}_{\Omega}(\varphi m)$ as follows
\begin{align*}
    \mathcal{J}_{\Omega}(\varphi m) &= \iint_{\Omega\times\Omega}j(x-y)|\varphi(x)m(x)-\varphi(y)m(y)|^2 \, dx\, dy\\
    &\leq 2\iint_{\Omega\times\Omega}j(x-y)|m(y)|^2|\varphi(x)-\varphi(y)|^2 \, dx\, dy \\
     & \qquad\qquad\qquad \qquad\qquad\qquad +
    2\iint_{\Omega\times\Omega}j(x-y)|\varphi(x)|^2|m(x)-m(y)|^2 \, dx\, dy\\
    &
    \leq2\,[\varphi]_{C^{0,1}(\overline{\Omega})}^2\int_\Omega|m(y)|^2\bigg(\int_\Omega|x-y|^2j(x-y)\,dx\bigg)\,dy
+2\,\|\varphi\|_{L^\infty(\Omega)}^2 \mathcal{J}_{\Omega}(m) 
\\
&
\leq 2\,(1 + \operatorname{diam}^2 (\Omega) )\,(1 + L_j) \, \|\varphi\|_{C^{0,1}(\bar{\Omega})}^2 \left(\|m\|_{L^2(\Omega)}^2 + \mathcal{J}_{\Omega}(m) \right)
\end{align*}
 with $L_j < + \infty$ from \ref{J2}. This proves the claim.
\end{proof}

\subsection*{Compactness argument: Proof of Theorem \ref{thm:Compactness} }\label{subsec:cmptc}
We are now in a position to prove our first main result.

\begin{proof}[Proof of Theorem \ref{thm:Compactness}]
By equations \eqref{eq:control_psi} and \eqref{eq:control_kappa}, we obtain
the uniform upper bound
\begin{equation}
  \sup_{h \in \mathbb{N}} \mathcal{J}_{\Omega} (m_h) < + \infty,
  \label{eq:unif-bound}
\end{equation}
and since $\mathcal{N}$ is compact, we also have
\begin{equation}
  {\sup_{h \in \mathbb{N}}}  \|m_h \|_{L^{\infty} (\Omega, \mathcal{N})} \leq
  c_{\mathcal{N}} \, \label{eq:Linfty_bound}
\end{equation}
for some constant $c_{\mathcal{N}}$ that depends only on the diameter of
$\mathcal{N}$ and its distance from the origin, e.g., $c_{\mathcal{N}} :=\text{diam} (\mathcal{N}) + \text{dist}(0,\mathcal{N})$. Let $(\Omega_k)_{k \in \mathbb{N}}$ be an exhaustion of $\Omega$ by relatively
compact open sets, meaning that $\Omega_k \Subset \Omega_{k + 1} \Subset
\Omega$ for every $k \in \mathbb{N}$ and $\cup_{k \in \mathbb{N}} \Omega_k =
\Omega$. We denote by $(\varphi_k)_k \subseteq C^{\infty}_c (\Omega)$ the
associated family of cut-off functions: for every $k \in \mathbb{N}$,
\[ 0 \leq \varphi_k \leq 1, \qquad \varphi_k \equiv 1 \textrm{ in } \Omega_k
   \qquad \mathrm{and} \qquad \varphi_k \equiv 0 \textrm{ in } \Omega
   \setminus \overline{\Omega_{k + 1}} . \]
For the proof, we will use the following scheme. For fixed $k \in
\mathbb{N}$
and an increasing sequence of natural numbers $(\nu (k, h))_{h \in
\mathbb{N}}$
we define the sequence of functions $(\tilde{m}_{\nu (k, h)})_{h \in
\mathbb{N}}$ as
\[ \tilde{m}_{\nu (k, h)} := \left\{ \begin{array}{ll}
     m_{\nu (k, h)} \varphi_k & \textrm{in } \Omega,\\
     0 & \textrm{in } \R^n \setminus \Omega .
   \end{array} \right. \]
Note that, by construction, $\tilde{m}_{\nu (k, h)} | \Omega_k = m_{\nu (k,
h)} | \Omega_k$. By first applying Lemma~\ref{lem:tagliato} and then
Lemma~\ref{lem:esteso}, we derive the estimate
\begin{align}
  \mathcal{J}_{\R^n} (\tilde{m}_{\nu (k, h)}) & \leq  \mathcal{J}_{\Omega} 
  (m_{\nu (k, h)} \varphi_k) + C \, L_j  \| \varphi_k \|_{C^{0, 1}
  (\bar{\Omega})}^2 \hspace{0.17em} \|m_{\nu (k, h)} \|_{L^2 (\Omega)}^2
  \nonumber\\
  & \leq  C \, (1 + L_j) \hspace{0.17em} \| \varphi_k \|_{C^{0, 1}
  (\bar{\Omega})}^2  \left( 2 \hspace{0.17em} \|m_{\nu (k, h)} \|_{L^2
  (\Omega)}^2 +\mathcal{J}_{\Omega} (m_{\nu (k, h)}) \right), \nonumber
\end{align}
for some $C = C (\Omega, \mathcal{N}, k) > 0$. Hence, from
\eqref{eq:unif-bound} and \eqref{eq:Linfty_bound}, we obtain that for any
fixed $k \in \mathbb{N}$ the following uniform bound holds:
\begin{equation}
  \sup_{h \in \mathbb{N}} \mathcal{J}_{\R^n} (\tilde{m}_{\nu (k, h)}) < +
  \infty . \label{eq:unifestforfixedk}
\end{equation}
Thus, for each fixed $k \in \mathbb{N}$, we use the compactness of the embedding
$\mathcal{X}^{j} (\R^n ; \R^{\ell}) \hookrightarrow
L^2_{\mathrm{loc}}(\R^n ; \R^{\ell})$ (see~{\cite[Theorem
1.1]{jarohs2020weth}}, which trivially extends to our vectorial case), from which we infer the
existence of a subsequence $(\tilde{\nu} (k, h))_{h \in
\mathbb{N}}$
of $(\nu (k, h))_{h \in
\mathbb{N}}$
and a function $\tilde{m}_k \in L^2_{\mathrm{loc}} (\R^n ; \R^{\ell})$ such that
\[ \tilde{m}_{\tilde{\nu} (k, h)} \to \tilde{m}_k \quad \textrm{strongly in }
   L^2_{\mathrm{loc}} (\R^n ; \R^{\ell}) \textrm{ and a.{\hspace{0.17em}}e. in }
   \R^n . \]
In particular, since $\tilde{m}_{\tilde{\nu} (k, h)} | \Omega_k =
m_{\tilde{\nu} (1, h)} | \Omega_k$, we have $\tilde{m}_k \in L^{\infty}
(\Omega_k ; \mathcal{N})$.

We apply iteratively this argument, starting with $k = 1$ and $\nu (1, h) =
h$, i.e., we consider the functions
\[ \tilde{m}_{\nu (1, h)} := \left\{ \begin{array}{ll}
     m_{\nu (1, h)} \varphi_1 & \textrm{in } \Omega,\\
     0 & \textrm{in } \R^n \setminus \Omega .
   \end{array} \right. \]
We thus obtain a subsequence $(\tilde{\nu} (1, h))_{h \in
\mathbb{N}}$
extracted from $(\nu (1, h) = h)_{h \in
\mathbb{N}}$
and a function $\tilde{m}_1 \in L^2_{\mathrm{loc}} (\R^n ; \R^{\ell})$, whose restriction to $\Omega_1$ is $\mathcal{N}$-valued, such that
$\tilde{m}_{\tilde{\nu} (1, h)} \to \tilde{m}_1$ strongly in $L^2_{\mathrm{loc}}
(\R^n ; \R^{\ell})$ and a.e.~in $\R^n$. However, $\tilde{m}_{\tilde{\nu} (1,
h)} | \Omega_1 = m_{\tilde{\nu} (1, h)}$, so that we obtain a subsequence
$(m_{\tilde{\nu} (1, h)})_{h \in
\mathbb{N}}$
of $(m_h)_{h \in
\mathbb{N}}$
such that
\[ m_{\tilde{\nu} (1, h)} \to \tilde{m}_1 \quad \textrm{strongly in } L^2
   (\Omega_1; \mathcal{N}) \textrm{ and a.e. in } \Omega_1 . \]
Next, we set $\nu (2, h) := \tilde{\nu} (1, h)$ and we reapply the
argument to the sequence $(\tilde{m}_{\nu (2, h)})_{h \in \mathbb{N}}$. Again,
arguing as before, we obtain a subsequence $(\tilde{\nu} (2, h))_{h \in
\mathbb{N}}$
of $(\nu (2, h))_{h \in
\mathbb{N}}$
and a function $\tilde{m}_2 \in L^2_{\mathrm{loc}} (\R^n ; \R^{\ell})$, whose restriction to $\Omega_2$ is $\mathcal{N}$-valued, such that
$\tilde{m}_{\tilde{\nu} (2, h)} \to \tilde{m}_2$ strongly in $L^2 (\R^n ;
\R^{\ell})$ and a.e. in $\R^n$. However $\tilde{m}_{\tilde{\nu} (2, h)} |
\Omega_2 = m_{\tilde{\nu} (2, h)} | \Omega_2$, so that we obtain a subsequence
$(m_{\tilde{\nu} (2, h)})_{h \in
\mathbb{N}}$
of $(m_{\tilde{\nu} (1, h)})_{h \in
\mathbb{N}}$
such that
\[ m_{\tilde{\nu} (2, h)} \to \tilde{m}_2 \quad \textrm{strongly in } L^2
   (\Omega_2 ; \mathcal{N}) \textrm{ and a.e. in } \Omega_2 .
\]
Moreover, by the uniqueness of the limit, since $\Omega_1 \Subset \Omega_2$,
we have
\[ \tilde{m}_{\tilde{\nu} (2, h)} | \Omega_1 = m_{\tilde{\nu} (1, h)} |
   \Omega_1 . \]
Proceeding iteratively in $k \in
\mathbb{N}$,
we deduce the existence of further subsequences $(m_{\nu (k, h)})_{h \in
\mathbb{N}}$ of $(m_{\nu (1, h)})_{h \in \mathbb{N}} = (m_h)_{h \in
\mathbb{N}}$,
such that $(m_{\nu (k, h)})_{h \in \mathbb{N}}$ is extracted from $(m_{\nu (k
- 1, h)})_{h \in \mathbb{N}}$, as well as a sequence $(\tilde{m}_k)_{k \in
\mathbb{N}}$
of $L^2_{\mathrm{loc}} (\R^n ; \R^l)$-functions,  such that
\[ m_{\tilde{\nu} (k, h)} \to \tilde{m}_k \quad \textrm{strongly in } L^2
   (\Omega_k ; \mathcal{N}) \textrm{ and a.e. in } \Omega_k\, , 
   \]
and such that
\[ \tilde{m}_{\tilde{\nu} (k, h)} | \Omega_j = m_{\tilde{\nu} (j, h)} |
   \Omega_j \quad \text{for every } 1 \leqslant j \leqslant k. \]
The previous compatibility condition assures that is well-defined the function
$m (x) = \tilde{m}_k (x)$ if $x \in \Omega_k$, that $m$ is
$\mathcal{N}$-valued and that $(m_{\tilde{\nu} (h, h)})_{h \in
\mathbb{N}}$
converges a.e.\ to $m$ in $\Omega$. Therefore, by the uniform bound
\eqref{eq:Linfty_bound} and the Dominated Convergence Theorem, we can conclude
that the diagonal sequence $(m_{\tilde{\nu} (h, h)})_{h \in
\mathbb{N}}$
strongly converges to $m$ in $L^2 (\Omega)$. Indeed, for every $k \in
\mathbb{N}$
and for every $h \geqslant k$ we have
\begin{align}
  \int_{\Omega} |
  m_{\tilde{\nu} (h, h)} - m |^2 & = 
  \int_{\Omega_k} |
  m_{\tilde{\nu} (h, h)} - m |^2 +
  \int_{\Omega
  \setminus \Omega_k} |
  m_{\tilde{\nu} (h, h)} - m |^2 \nonumber\\
  & = 
  \int_{\Omega_k} |
  m_{\tilde{\nu} (k, h)} - m |^2 +
  \int_{\Omega
  \setminus \Omega_k} |
  m_{\tilde{\nu} (h, h)} - m |^2 \nonumber\\
  & \leqslant 
  \int_{\Omega_k} |
  m_{\tilde{\nu} (k, h)} - m |^2 + 2 c_{\mathcal{N}} |
  \Omega \backslash
  \Omega_k |, \nonumber
\end{align}
so that by first taking the limit superior as $h \rightarrow \infty$, and then
the limit as $k \rightarrow \infty$, we conclude $(m_{\tilde{\nu} (h, h)})_{h
\in
\mathbb{N}}$
strongly converges to $m$ in $L^2 (\Omega)$.

It remains to show that $m \in \mathcal{X}^{K}
(\Omega ;
\mathcal{N})$. To this end, we use Fatou's Lemma and \eqref{eq:unif-bound} to
obtain
\[ 
   \mathcal{F}_{\Omega}
   (m) \leqslant \liminf_{h \rightarrow \infty}
   \mathcal{F}_{\Omega}
   (m_{\tilde{\nu} (h, h)}) \leqslant \sup_{h \in
   \mathbb{N}}
   \mathcal{F}_{\Omega}
   (m_h) < + \infty . \]
Thus, $m \in \mathcal{X}^{K}
(\Omega ;
\R^l)$. Finally, as $m_{\tilde{\nu} (h, h)} \in \mathcal{X}^{K}
(\Omega ;
\mathcal{N})$ for every $h \in \mathbb N$, and $\mathcal{N}$ is compact, the almost everywhere convergence yields $m  \in \mathcal{X}^{K}
(\Omega ;
\mathcal{N})$, and this concludes the proof of the theorem.
\end{proof}


\subsection*{Existence argument: Proof of Theorem~\ref{thm:existence}}\label{subsec:existence}

\begin{proof}[Proof of Theorem \ref{thm:existence}]
Given the compactness result stated in Theorem~\ref{thm:Compactness}, in essence, it remains to show that
$\mathcal{F}_{\Omega}$ is lower semicontinuous on the space $\mathcal{X}^K (\Omega ;
\mathcal{N})$ endowed with the strong $L^2$ topology.

Let $(m_k)_{k \in \mathbb{N}} \subset \mathcal{X}^K (\Omega ; \mathcal{N})$ be
a minimizing sequence for the functional $\mathcal{I}_{\Omega}
=\mathcal{F}_{\Omega} +\mathcal{P}_{\Omega}$. Since $\mathcal{P}_{\Omega}$ is
bounded from below, there exists a constant $C > 0$ such that $\lvert
\mathcal{I}_{\Omega} (m_k) \rvert \leq C$ for all $k \in \mathbb{N}$. By
Theorem \ref{thm:Compactness}, there exists, after possibly passing to a
subsequence, an element $m_{\ast} \in \mathcal{X}^K (\Omega ; \mathcal{N})$
such that $m_k \to m_{\ast}$ strongly in $L^2 (\Omega ; \mathbb{R}^{\ell})$.
The continuity of $\psi$, together with Fatou's lemma, then yields
\[ \mathcal{F}_{\Omega} (m_{\ast}) \leq \liminf_{k \to \infty}
   \mathcal{F}_{\Omega} (m_k) . \]
Furthermore, by our hypothesis, the term $\mathcal{P}_{\Omega}$ is continuous
for the strong $L^2$ topology, so that we also have
\[ \mathcal{I}_{\Omega} (m_{\ast}) \leq \liminf_{k \to \infty}
   \mathcal{I}_{\Omega} (m_k) = \inf_{m \in \mathcal{X}^K (\Omega ;
   \mathcal{N})} \mathcal{I}_{\Omega} (m) \leq \mathcal{I}_{\Omega} (m_{\ast})
   . \]
This chain of inequalities shows that $m_{\ast}$ is indeed a minimizer of
$\mathcal{I}_{\Omega}$ in $\mathcal{X}^K (\Omega ; \mathcal{N})$, which
completes the proof.
\end{proof}

\section{Micromagnetic case (Proofs of Theorem \ref{thm:small_bodies}, Theorem \ref{thm:large_bodies} and Theorem~\ref{prop:exist-constanti})}\label{Sec:proofs-mag}
Before proceeding with the proofs, we state a few preliminary ingredients. 

In what follows, we will often refer to the following Poincaré-type inequality (see~\cite[Theorem 8.2]{2012hitchhiker}), whose proof is a straightforward application of Jensen inequality.
\begin{lemma}[Poincaré-type Inequality]\label{lemma:Poincaré} Let $\Omega \subseteq \R^3$ be an open bounded set and $d_\Omega:= \operatorname{diam}(\Omega)$. Consider  $j:\R^3\to[0,+\infty]$ a Borel-measurable function, finite almost everywhere, and such that \ref{J4} holds, i.e., that $\QOm : = \essinf_{|z|< \operatorname{diam}(\Omega)}  j(z) > 0$.
Then
\begin{equation}\label{Ineq:Poincaré}
  \| m -  \langle m \rangle_{\Omega} \|^{2}_{L^2(\Omega;\R^3)}  \leq \frac{1}{ \QOm  \, |\Omega|} \iint_{\Omega \times \Omega} j(x-y)|m(x)-m(y)|^2 \, d x \,d y,  
\end{equation}
for every function $m \in L^2(\Omega; \R^3).$
\end{lemma}
\begin{proof}
It is enough to observe that 
\begin{align*}
 \int_{\Omega} | m -  \langle m \rangle_{\Omega} |^2 \, dx  &\leq \frac{1}{|\Omega|} \iint_{\Omega \times \Omega} | m(x) -  m(y) |^2 dx \, dy \\ &= \frac{1}{|\Omega|} \iint_{\Omega \times \Omega} \frac{j(x-y)}{j(x-y)} | m(x) -  m(y) |^2 dx \, dy \\ &\leq \frac{1}{ \QOm  \, |\Omega|} \iint_{\Omega \times \Omega} j(x-y)|m(x)-m(y)|^2 \, d x \,d y.
\end{align*}
This completes the proof.
\end{proof}

We also need two fundamental properties of the demagnetizing field, which play a pivotal role in characterizing the minimizers of the micromagnetic functional \eqref{eq:mainfunJpW}.

\begin{enumerate}[label=(P\arabic*)]
    \item \label{P1} 
    (\textit{Magnetostatic Inequality})
    Let $\Omega \subseteq \R^3$ be an open set. For every $m, u \in L^2(\Omega; \R^3)$ vanishing outside of $\Omega$, the magnetostatic self-energy satisfies
   \begin{equation}\label{rel:W}
    - \int_{\Omega} h_d[m] \cdot m \, d x \geq -2 \int_{\Omega} h_d[u] \cdot m \, d x + \int_{\Omega} h_d[u] \cdot u \, d x\, .
    \end{equation}
   Equivalently, in terms of the bilinear form
   \begin{equation}
   \mathcal{W}_{\Omega}(m, u)=- \int_{\Omega} h_d[u] \cdot m \, d x \, ,
    \end{equation}
   we have that
\begin{equation}\label{rel:WW2}
\mathcal{W}_{\Omega}(m, m)+\mathcal{W}_{\Omega}(u, u)-2 \, \mathcal{W}_{\Omega}(u, m) \geqslant 0.
\end{equation}
The inequality \eqref{rel:W} follows from the bilinearity and symmetry of $\mathcal{W}_{\Omega}(m, u)$, which stem from the self-adjointness of $h_d$ in $L^2$ and the positive semidefiniteness of $-h_d$. Indeed, one has
$$
-\int_{\mathbb{R}^3} h_d[m-u] \cdot(m-u)\, dx =\mathcal{W}_{\Omega}(m, m)+\mathcal{W}_{\Omega}(u, u)-2 \,\mathcal{W}_{\Omega}(m, u),
$$
from which, rearranging terms, \eqref{rel:W} follows. Equality holds if, and only if, $m-u$ is in the kernel of $h_d$. 
For further details, see \cite[Theorem 2]{di2020variational} and \cite{doi:10.1137/0512046}.
\smallskip
    \item \label{P2}
    (\textit{Uniform Single-Domain Property})
    Let $\Omega  \subseteq \R^3$ be an ellipsoid. If $m \in L^2(\R^3; \R^3)$ is constant within $\Omega$, then $h_d[m]\in L^2(\R^3; \R^3)$ is also constant within $\Omega.$ In particular, if $\Omega$ is a ball of radius $R$ and $m$ is constant in $\Omega$, then
    \begin{equation}\label{eq:unsdp}
    h_d[m]\bigchi_{B_R} = -\frac{1}{3} \langle m \rangle_{B_R} \bigchi_{B_R}\qquad \text{and}\qquad \mathcal{W}_\Omega (m) = \frac{1}{3} \|m\|^2_{L^2_{\Omega}}.
    \end{equation}
    This result is a direct consequence of the quadratic behavior of the Newtonian potential of a uniform distribution of charges or masses. Comprehensive details can be found in \cite{di2016newtonian, osborn1945demagnetizing, stoner1945xcvii}.
\end{enumerate}

\subsection*{Existence in the absence of compactness: Proof of Theorem~\ref{prop:exist-constanti}}

\begin{proof}
In view of \eqref{rel:W} and \eqref{eq:unsdp}, let us choose $u :=
\langle m \rangle_{B_R} \chi_{B_R}$, where $m \in L^2 (B_R ; \R^3)$ is
arbitrary. We immediately deduce that
\begin{equation}
  \label{bound:W_pt2} \mathcal{W}_{B_R} (m) \geq \frac{1}{3} |B_R |
  \hspace{0.17em} | \langle m \rangle_{B_R} |^2 \, .
\end{equation}
Let $C_R := 1 / (\mathcal{Q}_R |B_R |) > 0$ denote the constant appearing
in the Poincaré-type inequality \eqref{Ineq:Poincaré}. Then, for any $m \in
\mathcal{X}^j (B_R ; \Sf^2)$, inequality \eqref{bound:W_pt2} yields the lower
bound
\begin{equation}
  \label{eq:lower_bd_pf} \mathcal{E}_{B_R} (m) \geq \frac{1}{C_R}
  \hspace{0.17em} \|m - \langle m \rangle_{B_R} \|^2_{L^2 (B_R ; \R^3)} +
  \frac{1}{3} |B_R | \hspace{0.17em} | \langle m \rangle_{B_R} |^2 .
\end{equation}
Now, let $(m_h)_h \subseteq \mathcal{X}^j (B_R ; \Sf^2)$ be a minimizing
sequence. For any $\varepsilon > 0$, there exists $h (\varepsilon) \in
\mathbb{N}$ such that for every $h \geq h (\varepsilon)$ the following estimate holds
\[ \mathcal{E}_{B_R} (m_h) \leq \frac{\varepsilon}{3} |B_R | + \inf_{m \in
   \mathcal{X}^j (B_R ; \Sf^2)} \mathcal{E}_{B_R} (m) \leq
   \frac{\varepsilon}{3} |B_R | + \hspace{0.17em} \mathcal{E}_{B_R} (\sigma) =
   \frac{1 + \varepsilon}{3} |B_R |\, , \]
   where $\sigma$ is some constant function in $B_R$. 
In the last inequality, we used \eqref{eq:unsdp}.
Substituting this estimate into \eqref{eq:lower_bd_pf} yields that for every $h \geq h (\varepsilon)$ there holds
\[ \frac{1}{C_R} \hspace{0.17em} \|m_h - \langle m_h \rangle_{B_R} \|^2_{L^2
   (B_R ; \R^3)} + \frac{1}{3} |B_R | \hspace{0.17em} | \langle m_h
   \rangle_{B_R} |^2 \leq \frac{1 + \varepsilon}{3} |B_R |, \]
and rearranging the terms, we obtain
\begin{equation}
  \label{bound:1_2} \frac{1}{C_R}  \hspace{0.17em} \|m_h - \langle m_h
  \rangle_{B_R} \|^2_{L^2 (B_R ; \R^3)} \leq \frac{1}{3} |B_R |  (1 - |
  \langle m_h \rangle_{B_R} |^2 + \varepsilon) \hspace{0.17em} .
\end{equation}
Due to the unitary-norm constraint $|m_h | \equiv 1$, it follows that
\begin{equation}
  \label{eq:2_2} \|m_h - \langle m_h \rangle_{B_R} \|^2_{L^2 (B_R ; \R^3)} =
  |B_R |  (1 - | \langle m_h \rangle_{B_R} |^2) .
\end{equation}
Thus, combining \eqref{bound:1_2} and \eqref{eq:2_2} we deduce that for every $h \geq h (\varepsilon)$
\[ \|m_h - \langle m_h \rangle_{B_R} \|^2_{L^2 (B_R ; \R^3)}  \left(
   \frac{1}{C_R} - \frac{1}{3} \right) \leq \frac{\varepsilon}{3} |B_R |\, . 
\]
Since by assumption $C_R < 3$, we conclude
that for every $h \geq h (\varepsilon)$
\[ \|m_h - \langle m_h \rangle_{B_R} \|^2_{L^2 (B_R ; \R^3)} \leq \frac{C_R}{3
   - C_R} |B_R |  \hspace{0.17em} \varepsilon, \]
which implies that
\[ \lim_{h \to \infty} \|m_h - \langle m_h \rangle_{B_R} \|^2_{L^2 (B_R ;
   \R^3)} = 0. \]
Now notice that since $|m_h | \equiv 1$, we have $| \langle m_h \rangle_{B_R}
| \leq 1$, i.e., the sequence $(\langle m_h \rangle_{B_R})_h$, takes values in
the unit ball of $\R^3$. Hence there exist a point $m_0 \in \mathbb{R}^3$,
with $|m_0 | \leq 1$, such that, maybe up to a subsequence, $\langle m_h
\rangle_{B_R} \to m_0$. Clearly, we have also
\[ \| \langle m_h \rangle_{B_R} - m_0 \|_{L^2 (B_R ; \mathbb{R}^3)}^2 = |B_R |
   \cdot | \langle m_h \rangle_{B_R} - m_0 |^2 \rightarrow 0. \]
Therefore, by the triangle inequality
\[ \|m_h - m_0 \|_{L^2 (B_R ; \R^3)} \leq \|m_h - \langle m_h \rangle_{B_R}
   \|_{L^2 (B_R ; \R^3)} + \| \langle m_h \rangle_{B_R} - m_0 \|_{L^2 (B_R ;
   \mathbb{R}^3)}, \]
we find that, up to a subsequence, $m_h \to m_0$ strongly in $L^2 (B_R ;
\mathbb{R}^3)$ and $m_h \to m_0$ almost everywhere in $B_R$. Since $|m_h |
\equiv 1$, this implies that $|m_0 | = 1$. Hence, the constant function $m_0$
belongs to the space $\mathcal{X}^j (B_R ; \mathbb{S}^2)$ and is thus an
admissible candidate for the minimization of $\mathcal{E}_{B_R}$. By Fatou's
Lemma and recalling that $\mathcal{W}_{B_R}$ is continuous with respect to the
strong $L^2$-convergence, we conclude that
\[ \inf_{m \in \mathcal{X}^j (B_R ; \Sf^2)} \mathcal{E}_{B_R} (m) \leq
   \mathcal{E}_{B_R} (m_0) \leq \liminf_{h \to \infty} \mathcal{E}_{B_R} (m_h)
   = \inf_{m \in \mathcal{X}^j (B_R ; \Sf^2)} \mathcal{E}_{B_R} (m) . \]
We have thus proved that there exists at least a minimizer of
$\mathcal{E}_{B_R}$ in $\mathcal{X}^j (B_R ; \Sf^2)$, which is given by a
constant function.

We are left to prove that any other minimizer must also be constant. For this,
it is enough to observe that if $\tilde{m} \in \mathcal{X}^j (B_R ;
\mathbb{S}^2)$ is a minimizer of $\mathcal{E}_{B_R}$, then we can choose the
constant sequence $m_h = \tilde{m}$ as a minimizing sequence. The above
argument then ensures that $\tilde{m}$ must be constant, thus concluding the
proof of the theorem.
\end{proof}

We can now derive  Theorem \ref{thm:small_bodies} as a consequence of Theorem~\ref{prop:exist-constanti}.

\subsection*{Small bodies case: Proof of Theorem \ref{thm:small_bodies}}
\begin{proof}
Under Assumption~\ref{J5}, there exist parameters $R_0 > 0$, $s \in (0, 1)$, and $C > 0$ such that the interaction kernel $j$ satisfies
\[ j (z) \geq \frac{C}{|z|^{3 + 2 s}} \quad \text{for all } |z| < R_0 . \]
For any radius $R \leq \frac{1}{2} R_0$, the kernel $j$ also satisfies
Assumption~\ref{J4} within the spherical domain $\Omega = B_R$. Specifically, the essential infimum $Q_{B_R}$ defined in \eqref{cond:inf} can be bounded from below as
\[ Q_{B_R} \geq \essinf_{z \in B_{2 R}} \frac{C}{|z|^{3 + 2 s}} =
   \frac{C}{2^{3 + 2 s}} \cdot \frac{1}{R^{3 + 2 s}} . \]
To apply Theorem~\ref{prop:exist-constanti}, we require the Poincaré constant \( C_R := \frac{1}{Q_{B_R} |B_R|} \) to satisfy \( C_R < 3 \). This condition is equivalent to  
\[
Q_{B_R} > \frac{1}{4\pi R^3}.
\]  
To guarantee this inequality, we define the critical radius  
\[
R^* := \min\left\{\frac{1}{2}R_0, \left(\frac{4\pi C}{2^{3+2s}}\right)^{1/2s}\right\}.
\]  
For all \( R \leq R^* \), the lower bound on \( Q_{B_R} \) yields  
\[
Q_{B_R} \geq \frac{C}{2^{3+2s}} \cdot \frac{1}{R^{3+2s}} > \frac{1}{4\pi R^3}.
\]  
This ensures \( C_R < 3 \), thereby satisfying the hypothesis of Theorem~\ref{prop:exist-constanti}. Consequently, all minimizers of the energy functional \( \mathcal{E}_{B_R} \) must be constant configurations. The critical radius \( R^* \) explicitly depends on the parameters \( R_0 \), \( C \), and \( s \) characterizing the kernel \( j \), reflecting the interplay between nonlocal exchange interactions and domain size in the small-body regime.  

This concludes the proof that uniform magnetization states are energetically favorable for spherical domains with radii below the critical threshold \( R^* \).
\end{proof}

\subsection*{Large bodies case: Proof of Theorem \ref{thm:large_bodies}}
We conclude with the proof of Theorem \ref{thm:large_bodies} by adopting an approach similar to the one used by Brown in~\cite[Sec.~5] {brown1969fundamental}, suitably modified to fit our total nonlocal framework.

Before proceeding with the proof, we state a preliminary lemma in which we demonstrate that vortex-type configurations are always energetically preferred by the stray field over constant configurations. 
In what follows, we denote by $(e_1, e_2, e_3)$ an orthonormal frame, and for every $x \in \mathbb{R}^3$ we write $x = x_{\perp} + x_3 e_3$, where $x_{\perp}$ is the projection of $x$ onto the plane spanned by $(e_1,e_2)$, i.e., $x_{\perp} = x_1 e_1 + x_2 e_2$.
\begin{lemma}\label{Lemma-Magnetos}
For every $R > 0$, consider the vortex-type configuration $$m = m_1 e_1 + m_2 e_2 + m_3 e_3$$ defined on $B_R$ by
\begin{equation}\label{m2}
    m_1(x) = -\frac{x_2}{R} \sqrt{2 - \frac{|x_{\perp}|^2}{R^2}}, \quad 
    m_2(x) = \frac{x_1}{R} \sqrt{2 - \frac{|x_{\perp}|^2}{R^2}}, \quad 
    m_3(x) = 1 - \frac{|x_{\perp}|^2}{R^2}.
\end{equation}
There exists a constant $c_2 > 0$ such that, for any constant configuration $\sigma \in L^2(B_R; \mathbb{S}^2)$, the following holds:
\begin{equation}\label{thesis}
    \mathcal{W}_{B_R}(m) - \mathcal{W}_{B_R}(\sigma) = -c_2 R^3,
\end{equation}
where $\mathcal{W}_{B_R}$ represents the magnetostatic self-energy, as defined in \eqref{eq:intro_magnetostatic}.
\end{lemma}

\begin{proof}
For any configuration $m^* \in C^1(\bar{B_R}; \Sf^2)$ such that $\operatorname{div}m^* = 0$, we obtain from Divergence Theorem that
\begin{equation}\label{eq:div}
    \mathcal{W}_{B_R}(m^*) = - \int_{B_R} m^*(x) \cdot (-\nabla \Phi(x))\, d x = \int_{\partial B_R}  \Phi(x) m^*(x) \cdot n \, d S(x),
\end{equation}
where, by \eqref{eq:Phi},
\begin{equation}\label{eq:Phi_m}
    \Phi(x) = \frac{1}{4\pi}  \int_{\partial B_R} \frac{m^*(y) \cdot n(y)}{|x - y|} d S (y), \quad \text{for every } x \in B_R.
\end{equation}
Therefore 
\begin{equation}\label{new:mag}
 \mathcal{W}_{B_R}(m^*) = \frac{1}{4\pi}  \iint_{\partial B_R \times \partial B_R} \frac{\left(m^*(x) \cdot n(x)\right) \left(m^*(y) \cdot n(y)\right)}{|x - y|} d S(y) \, d S(x).
\end{equation}
Consequently, the configuration $m$ in \eqref{m2} and any constant configuration satisfy the integral representation \eqref{new:mag}.
Additionally, by the \textit{Uniform Single-Domain Property} \ref{P2}, every configuration that is constant in $B_R$ has the same energy, i.e., for any $\sigma \in\Sf^2$ there holds
\begin{equation}
\mathcal{W}_{B_R}(\sigma) =  \frac{4}{9}  \pi R^3.  
\end{equation}
It is then sufficient to prove \eqref{thesis} for at least one constant configuration.
Consider $\sigma = e_3$, and from \eqref{new:mag} we infer that
\begin{align}\label{change1}
\begin{split}
  \mathcal{W}_{B_R}(e_3) &=  \frac{1}{4\pi}\iint_{\partial B_R \times \partial B_R} \frac{1}{R^2} \frac{x_3 y_3}{|x - y|} d S(y) \, d S(x) \\ & = \frac{R^3}{4\pi} \iint_{\Sf^2 \times \Sf^2} \frac{x_3 y_3}{|x - y|} d S(y) \, d S(x)  \\ & = R^3 \, \mathcal{W}_{B_1}(e_3), 
\end{split}
\end{align}
where in \eqref{change1} we applied the change of variables $(x,y) \mapsto (Rx, Ry)$ for $(x,y) \in B_R \times B_R$.

Looking now at the configuration $m$, we have
\begin{align}\label{change2}
  \mathcal{W}_{B_R}(m) &=\frac{1}{4\pi} \iint_{\partial B_R \times \partial B_R} \frac{\left(m(x) \cdot n(x)\right) \left(m(y) \cdot n(y)\right)}{|x - y|} d S(y) \, d S(x) \notag \\ 
  &= \frac{R^3}{4\pi}  \iint_{\Sf^2 \times \Sf^2} \frac{\left(m(Rx) \cdot n(Rx)\right) \left(m(Ry) \cdot n(Ry)\right)}{|x - y|} d S(y) \, d S(x) \notag \\ 
  &= \frac{R^3}{4\pi}  \iint_{\Sf^2 \times \Sf^2} \frac{\left(m_{\bullet}(x) \cdot n(Rx)\right) \left(m_{\bullet}(y) \cdot n(Ry)\right)}{|x - y|} d S(y) \, d S(x) \notag \\ 
  &= \frac{R^3}{4\pi}  \iint_{\Sf^2 \times \Sf^2} \frac{x^3_3 y^3_3}{|x - y|} d S(y) \, d S(x) \notag \\ 
  &= \mathcal{W}_{B_1}(m_{\bullet}),
\end{align}
where in \eqref{change2} we applied again the change of variables $(x,y) \mapsto (Rx, Ry)$ for $(x,y) \in B_R \times B_R$ and $m_{\bullet}$ is the rescaled configuration $m_{\bullet}(x ) = m(R x)$ for every $x \in B_1$ with coefficients 
\begin{equation}\label{coeff_mbul}
m_{\bullet, 1}(x)=-x_2 \sqrt{2-|x_{\perp}|^2}, \quad m_{\bullet, 2}(x)=x_1 \sqrt{2-|x_{\perp}|^2}, \quad m_{\bullet, 3}(x)=1-|x_{\perp}|^2. 
\end{equation}  Here $m_{\bullet}(x) \cdot n(Rx) = x_3 (1 - |x_{\perp}|^2) = x^3_3$ and, analogously, $m_{\bullet}(y) \cdot n(Ry) = y^3_3$.
We introduce
\begin{align}
    \Sf^2_+ = \{ \xi \in \Sf^2 : \xi_3 \geq 0  \} \quad \text{and} \quad \Sf^2_- = \{ \xi \in \Sf^2 : \xi_3 \leq 0 \},
\end{align}
and we decompose $\mathcal{W}_{B_1}(m_{\bullet})$ as follows:
\begin{align}\label{decompose}
  \mathcal{W}_{B_1}(m_{\bullet}) &= \frac{1}{4\pi} \iint_{\Sf^2 \times \Sf^2} \frac{x^3_3 y^3_3}{|x - y|} d S(y) \, d S(x)  \notag \\ 
  &= \frac{1}{4\pi}\int_{\Sf^2_+} \int_{\Sf^2} \frac{x^3_3 y^3_3}{|x - y|} d S(y) \, d S(x)   + \frac{1}{4\pi} \int_{\Sf^2_-} \int_{\Sf^2} \frac{x^3_3 y^3_3}{|x - y|} d S(y) \, d S(x) \notag \\ 
  &=  \frac{1}{4\pi} \int_{\Sf^2_+} \int_{\Sf^2} \frac{x^3_3 y^3_3}{|x - y|} d S(y) \, d S(x) \notag\\
  &\qquad \qquad\qquad +\frac{1}{4\pi}\int_{\Sf^2_+} \int_{\Sf^2} \frac{x^3_3 y^3_3}{|x_{\perp} - x_3 e_3 - y_{\perp} + y_3 e_3|} d S(y) \, d S(x) \notag \\ 
  &= \frac{2}{4\pi}  \int_{\Sf^2_+} \int_{\Sf^2} \frac{x^3_3 y^3_3}{|x - y|} d S(y) \, d S(x),
\end{align}
where we applied first the change of variable $(x_1,x_2, x_3 ) \mapsto (x_1, x_2, - x_3) =: \hat{x} $ for $x \in \Sf^2_-$ and then $(y_1,y_2, y_3 ) \mapsto (y_1, y_2, - y_3) =: \hat{y} $ for $y \in \Sf^2$. Note that $x_{\perp} - x_3 e_3 = \hat{x}$ and $y_{\perp} - y_3 e_3 = \hat{y}. $
We apply the same strategy for the second integral, and we infer
\begin{align}\label{decompose-2}
\begin{split}
\mathcal{W}_{B_1}(m_{\bullet})  &= \frac{2}{4\pi}   \int_{\Sf^2_+} \int_{\Sf^2} \frac{x^3_3 y^3_3}{|x - y|} d S(y) \, d S(x) \\    &=  \frac{2}{4\pi}\int_{\Sf^2_+} \int_{\Sf^2_+} \frac{x^3_3 y^3_3}{|x - y|} d S(y) \, d S(x)   + \frac{2}{4\pi}\int_{\Sf^2_+} \int_{\Sf^2_-} \frac{x^3_3 y^3_3}{|x - y|} d S(y) \, d S(x) \\ &=  \frac{2}{4\pi} \iint_{\Sf^2_+ \times \Sf^2_+} x^3_3 y^3_3 \left(\frac{1}{|x - y|} - \frac{1}{|x - y_{\perp} + y_3 e_3| } \right) d S(y) \, d S(x),
\end{split}
\end{align}
where in \eqref{decompose-2} for the second integral we applied $(y_1,y_2, y_3 ) \mapsto (y_1, y_2, - y_3)$ for $y \in \Sf^2_-$.
We then observe that
\begin{align}\label{frazione2}
 \frac{1}{|x - y|} - \frac{1}{|x - y_{\perp} + y_3 e_3| } = \frac{|x - y_{\perp} + y_3 e_3|^2 - |x - y|^2}{\operatorname{\omega}(x,y)} = \frac{4 x_3 y_3}{\operatorname{\omega}(x,y)},   
\end{align}
with 
\begin{equation*}
\operatorname{\omega}(x,y) := |x - y| \, |x - y_{\perp} + y_3 e_3| \left( |x - y_{\perp} + y_3 e_3| + |x - y|  \right). 
\end{equation*}
Finally, from \eqref{change2}, \eqref{decompose-2} and \eqref{frazione2}, we obtain that 
\begin{equation}
\mathcal{W}_{B_R}(m) = \frac{2 R^3}{\pi} \iint_{\Sf^2_+ \times \Sf^2_+} \frac{x^4_3 y^4_3}{\operatorname{\omega}(x,y)} d S(y) \, d S(x).
\end{equation}
Operating analogously for the constant configuration $\sigma = e_3$ in \eqref{change1}, we also infer
\begin{align}
\mathcal{W}_{B_R}(e_3) & = R^3 \, \mathcal{W}_{B_1}(e_3)  \notag \\
& = \frac{R^3}{4\pi} \iint_{\Sf^2 \times \Sf^2} \frac{x_3 y_3}{|x - y|} d S(y) \, d S(x)  \notag \\ 
& = \frac{2 R^3}{\pi}  \iint_{\Sf^2_+ \times \Sf^2_+} \frac{x^2_3 y^2_3}{\operatorname{\omega}(x,y)} d S(y) \, d S(x).
\end{align}
Since $(x,y) \in \Sf^2_+ \times \Sf^2_+ $, we have that $x^4_3 y^4_3 \leq x^2_3 y^2_3$ in  $(x,y) \in \Sf^2_+ \times \Sf^2_+ $ and therefore, by setting $c_2:=\mathcal{W}_{B_1}(e_3)- \mathcal{W}_{B_1}(m_{\bullet})>0$ we get \eqref{thesis}. This concludes the proof.
\end{proof}

We are now ready to prove Theorem  \ref{thm:large_bodies}.
 
\begin{proof}[Proof of Theorem \ref{thm:large_bodies}]

By Theorem \ref{thm:existence-magn}, we know that for any radius $R > 0$,
there exists at least one minimizer of the energy functional $\mathcal{E}_{B_R}$. Our objective is to demonstrate that for sufficiently large ferromagnetic particles, a non-uniform magnetization configuration
attains lower energy than any constant magnetization state.

To this end, we consider the vortex-type configuration introduced in Lemma
\ref{Lemma-Magnetos}. Specifically, we define the magnetization field $m : B_R
\to \mathbb{S}^2$ as
\begin{equation}
  m_1 (x) = - \frac{x_2}{R}  \sqrt{2 - \frac{|x_{\perp} |^2}{R^2}}, \quad m_2
  (x) = \frac{x_1}{R}  \sqrt{2 - \frac{|x_{\perp} |^2}{R^2}}, \quad m_3 (x) =
  1 - \frac{|x_{\perp} |^2}{R^2} .
\end{equation}
We will also work with the rescaled configuration
$m_{\bullet} : B_1 \to \mathbb{S}^2$, defined as
\begin{equation}\label{def:m_bullet 1}
m_{\bullet}(x ) = m(R x), \end{equation}
which explicitly reads as
\begin{equation}\label{def:m_bullet 3}
m_{\bullet, 1}(x)=-x_2 \sqrt{2-|x_{\perp}|^2}, \quad m_{\bullet, 2}(x)=x_1 \sqrt{2-|x_{\perp}|^2}, \quad m_{\bullet, 3}(x)=1-|x_{\perp}|^2 ,  
\end{equation}  
whose plot is represented in Figure~\ref{fig:mB1}.
\begin{figure}[t]
    \centering
    \includegraphics[width=\linewidth]{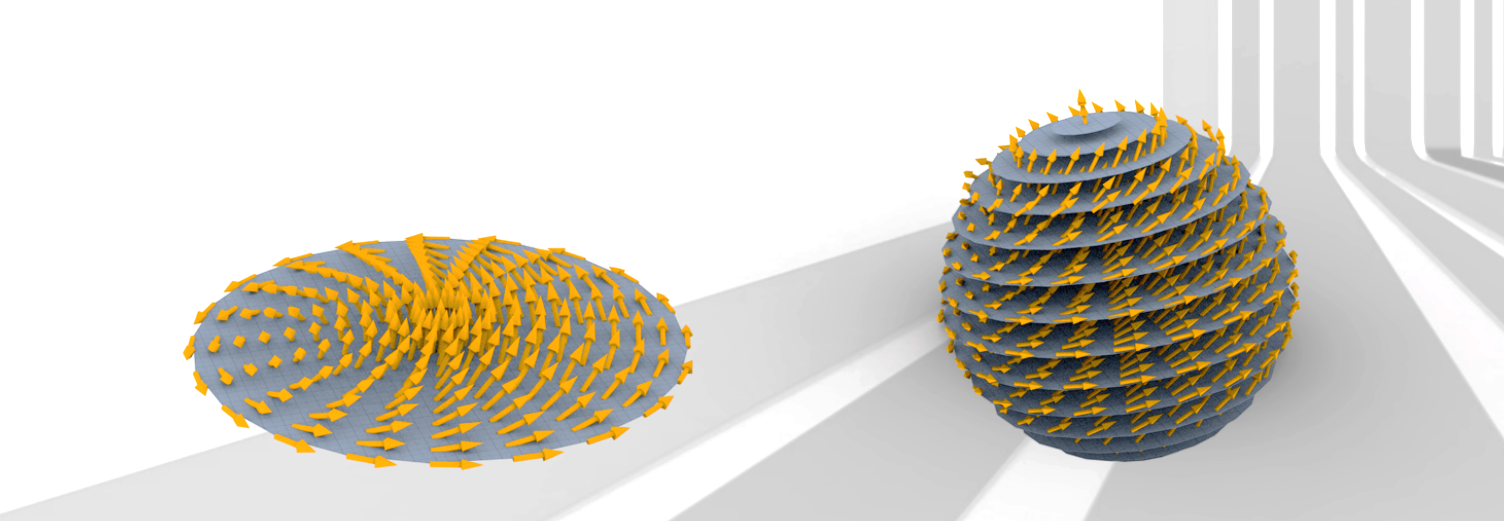}
    \caption{On the left, the rescaled configuration $m_{\bullet}$ restricted to the equatorial plane. On the right, a vector plot of $m_{\bullet}$ over the slice coordinate planes stacked along the $z$-axis.}
    \label{fig:mB1}
\end{figure}
It is readily verified that $m$ is $\mathbb{S}^2$-valued and belongs to
$C^{\infty} (B_R)$. Also, a straightforward computation shows that $\operatorname{div} m=0$. However, since $m$ is not tangential to $\partial B_R$, the associated
demagnetizing field $h_d [m]$ does not vanish due to surface magnetic charges.
By Property \ref{P2}, the energy associated with any constant configuration
$\sigma \in L^2 (B_R ; \mathbb{S}^2)$ is given by
\begin{equation}
  \label{reference-state} \mathcal{E}_{B_R} (\sigma) =\mathcal{W}_{B_R}
  (\sigma) = \frac{4}{9} \pi R^3 .
\end{equation}
We use this as our \textit{reference energy} and compare it with the energy of $m$.

To estimate the nonlocal exchange energy $\mathcal{J}_{B_R} (m)$ without
computing the double integral explicitly, we utilize the following estimate
\begin{align}
\mathcal{J}_{B_R}(m)  &= R^6 \iint_{B_1 \times B_1} j(Rx-Ry) \, |m(Rx)-m(Ry)|^2\,dx \, dy \label{eq:rescale} \\ 
&= R^6 \iint_{B_1 \times B_1} j(Rx-Ry) \, |m_{\bullet}(x)-m_{\bullet}(y)|^2\,dx \, dy  \\
&=  R^6  \int_{B_1} \left( \int_{B_1 - x} j(Rh) \,|\Tilde{m}_{\bullet}(x + h)-\Tilde{m}_{\bullet}(x)|^2 \,dh \right) dx \label{eq:transla} \\ 
&\leq  R^6 \int_{B_{2}} j(Rh) \left( \int_{\R^3} |\Tilde{m}_{\bullet}(x + h)-\Tilde{m}_{\bullet}(x)|^2 \,dx \right) dh, \label{ineq:2palla}
\end{align}
where we applied the change of variables $(x,y) \mapsto (Rx, Ry)$ for $(x,y) \in B_R \times B_R$ in \eqref{eq:rescale} and  $y \mapsto x - h$ for fixed $x \in B_1$ in \eqref{eq:transla}, and we used the fact that $j$ is symmetric by Assumption \ref{J1}.
Since from definition \eqref{def:m_bullet 1}--\eqref{def:m_bullet 3} $m_{\bullet}$ is not well-defined outside $B_{\sqrt{2}}$, here we denote by $\Tilde{m}_{\bullet}$ an appropriate extension on $H^1(\R^3; \R^3).$
For the inner integral in \eqref{ineq:2palla}, we apply the Fundamental Theorem of Calculus for Sobolev functions (see, e.g.,~\cite[Prop.~9.3]{brezis2011functional}), yielding 
\begin{align}\label{ineq:Sobolev}
  \int_{\R^3} |\Tilde{m}_{\bullet}(x + h)-\Tilde{m}_{\bullet}(x)|^2 \,dx &\leq |h|^2  \| \nabla \Tilde{m}_{\bullet} \|^{2}_{L^2(\R^3;\R^{3\times 3})} \\ & \leq \Tilde{c} \, |h|^2 \| m_{\bullet}\|^{2}_{H^1(B_1; \R^3)},\label{estim:extension}
\end{align}
for some $\Tilde{c} > 0$, where in \eqref{estim:extension} we used an extension result for Sobolev functions (see, e.g.,~\cite[Theorem 9.7]{brezis2011functional}).
From \eqref{ineq:2palla} and \eqref{estim:extension}, we deduce that 
\begin{align}\label{ineq:nonloc}
\mathcal{J}_{B_R}(m) &\leq \tilde{c} \, \left(R^4 \int_{B_{2}} j(Rh) |Rh|^2 dh \right) \|  m_\bullet \|^{2}_{H^1(B_1;\R^{3})} \notag\\ 
&=  \tilde{c} \, R \left(\int_{B_{2R}} j(h) |h|^2 dh \right) \|  m_\bullet \|^{2}_{H^1(B_1;\R^{3})}. 
\end{align}
The nonlocal exchange energy $\mathcal{J}_{B_R}$ is then estimated as
\begin{equation}\label{finalconto:grad}
\mathcal{J}_{B_R}(m) \leq  c_1 R \left(\int_{B_{2R}} j(h) |h|^2 dh \right), 
\end{equation}
with $c_1 = \tilde{c} \, \| m_{\bullet}\|^{2}_{H^1(B_1; \R^3)}$ --- a simple computation gives
$\| m_{\bullet}\|^{2}_{H^1(B_1; \R^3)}=\left|B_1\right|+\frac{4}{15}(68-15 \pi) \pi=\frac{4}{15}\pi(73-15 \pi)$.
From Lemma \ref{Lemma-Magnetos}, we know that $\mathcal{W}_{B_R}(m)-\mathcal{W}_{B_R}(\sigma)= - c_2\, R^3$ for some $ c_2  > 0$. Hence, comparing the total energy with the reference state \eqref{reference-state}, we get
\begin{equation*}
\mathcal{E}_{B_R}(m)-\mathcal{E}_{B_R}(\sigma) \leq c_1 R \left(\int_{B_{2R}} j(h) |h|^2 dh \right) -c_2 \, R^3.
\end{equation*}
For sufficiently large $R$, the right-hand side is negative due to Lemma
\ref{lemma:j6}, implying that the non-uniform configuration has a lower energy than the constant state. Thus, there exists a critical radius $R^{\ast \ast}$,  which depends on the interaction kernel $j$, beyond which every minimizer is non-constant.
\end{proof}

\subsection*{Author Contributions} \noindent G. Di Fratta, R. Giorgio, and L. Lombardini contributed equally to all the results of this article.

\subsection*{Acknowledgements}
{\sc{G.DiF.}} thanks TU Wien and MedUni Wien for their hospitality. {\sc{G.DiF.}} is a member of Gruppo Nazionale per l'Analisi Matematica, la Probabilità e le loro Applicazioni (GNAMPA) of INdAM. 

\subsection*{Funding}
\noindent{\sc {G.DiF.}} and {\sc R.G.} are partially supported by the Austrian Science Fund (FWF) through the project {\emph{Analysis and Modeling of Magnetic Skyrmions}} (grant 10.55776/P34609). 
\noindent{\sc{G.DiF.}} is partially supported by the Italian Ministry of Education and Research through the PRIN2022 project {\emph{Variational Analysis of Complex Systems in Material Science, Physics and Biology}} No.~2022HKBF5C. 

\subsection*{Data Availability} \noindent Data sharing does not apply to this article as no datasets were generated or analyzed during this study.

\subsection*{Conflict of Interest} The authors declare no conflict of interest.

\bibliographystyle{siam} 
\bibliography{Brown-bibliography.bib}
\end{document}